\documentclass[reqno,oneside]{amsart}
\usepackage{amsmath, amsthm}
\usepackage{amssymb}
\usepackage{stackengine}
\usepackage{xcolor}
\usepackage{appendix}
\usepackage[a4paper, total={6in, 9.5in}]{geometry}

\newtheorem{theorem}{Theorem}[section]
\newtheorem{lemma}[theorem]{Lemma}
\newtheorem{proposition}[theorem]{Proposition}
\theoremstyle{definition}
\newtheorem{definition}[theorem]{Definition}
\newtheorem{corollary}{Corollary}[theorem]

\theoremstyle{remark}
\newtheorem{remark}[theorem]{Remark}

\numberwithin{equation}{section}

\def\XXint#1#2#3{{\setbox0=\hbox{$#1{#2#3}{\int}$}
     \vcenter{\hbox{$#2#3$}}\kern-.5\wd0}}

\newcommand{\Om} {\Omega}
\newcommand{\sr}{\mathbb{R}}
\newcommand{\coa}{C^{1,\alpha}(\overline{\Om})}
\newcommand{\De} {\Delta}
\newcommand{\p}{\Delta_p}
\newcommand{\s}{(-\Delta)^s}
\newcommand{\qs}{(-\Delta_q)^s}
\newcommand{\ps}{(-\Delta_p)^s}
\newcommand{\ba}{\beta}
\newcommand{\bdry}{\mathbb{R}^N \setminus \Omega}
\newcommand{\pa}{\partial}
\newcommand{\rar}{\rightarrow}
\newcommand{\na} {\nabla}
\newcommand{\omc}{(\overline{\Om})}
\newcommand{\sn}{\mathbb{N}}
\newcommand{\si}{\sigma}

\newcommand{\ball}{B_{\epsilon} (z_0)}
\newcommand{\intcd}{int\ C_d^+ (U)}
\newcommand{\li}{\mathcal{L}_0^{-1}}

\newcommand{\xo}{\mathbb{X}_0}
\newcommand{\ul}{\underline{u}}

\newcommand{\ti}{\underline{I}_{\la}}
\newcommand{\wop}{W^{1,p}_0 (\Omega)}

\newcommand{\il}{\underline{I}}
\newcommand{\ue}{U_{\varepsilon}}
\newcommand{\la} {\lambda}
\newcommand{\test}{C_c^{\infty} (\Om)}

\begin{document}

\title[Global Multiplicity and Comparison Principles]{Global Multiplicity and Comparison Principles for Singular Problems driven by Mixed Local-Nonlocal Operators}

\author{R. Dhanya}
\email{dhanya.tr@iisertvm.ac.in}

\author{Sarbani Pramanik}
\email{sarbanipramanik20@iisertvm.ac.in}

\address{School of Mathematics, Indian Institute of Science Education and Research Thiruvananthapuram, Maruthamala, Thiruvananthapuram, Kerala, 695551, India.}



\begin{abstract}

We study a singular elliptic problem driven by a mixed local-nonlocal operator of the form
\begin{equation*}
    \begin{aligned}
        -\p u + \qs u &= \frac{\la}{u^{\delta}} + u^r \text{ in } \Om\\
        u>0 \text{ in } \Om,\ u &=0 \text{ in } \bdry
    \end{aligned}
\end{equation*}
where $p > sq$, $0<\delta<1$ and $\lambda > 0$ is a parameter. The nonlinearity exhibits a singular power-type behavior near zero and displays at most a critical growth at infinity. We establish a global multiplicity result with respect to the parameter $\lambda$ by identifying a sharp threshold that separates existence, non-existence, and multiplicity regimes, a result that is new for singular problems involving mixed local-nonlocal operators. We also derive a Hopf-type strong comparison principle adapted to this nonlinear setting, which provides the main analytical tool for the global multiplicity result. Additionally, we investigate qualitative properties of solutions that are essential for the variational analysis, such as a uniform $L^{\infty}$-estimate and a Sobolev versus H\"older local minimizer result. The analytical tools developed herein are of independent mathematical interest, with their applicability extending over a broader class of mixed local-nonlocal problems.

\end{abstract}

\maketitle

\section{Introduction}

We consider the following singular mixed local-nonlocal problem:
\begin{equation}\tag{$P_{\la}$}\label{P_lambda}
    \begin{aligned}
        -\p u + \qs u &= \frac{\la}{u^{\delta}} + u^r \text{ in } \Om\\
        u>0 \text{ in } \Om,\ u &=0 \text{ in } \bdry
    \end{aligned}
\end{equation}
where $\Om \subset \sr^N$, $N\geq 2$, denotes a bounded smooth domain. Here, $p,q \in (1, \infty)$, $s\in (0,1)$ with  $p>sq$ and $0< \delta <1$. Also, $\max \{ p-1, q-1\} < r \leq r^*= \max \{p^* -1,\ q_s^* -1 \}$ and $\la>0$ is a parameter. The operator consists of the $p$-Laplacian and the fractional $q$-Laplacian, while the right-hand side combines a singular term with a superlinear, possibly critical, power-type nonlinearity. Problems of this form naturally arise in models incorporating both local diffusion and long-range interactions, and the mixed nature of the operator, together with the presence of a singular nonlinearity, leads to substantial analytical challenges.

The problem \eqref{P_lambda} is rooted in the classical convex-concave type equations
\begin{align}\label{eqn1.1}
    -\De u = \la u^{\sigma} + u^r \text{ in } \Om,\ u=0 \text{ on } \partial \Omega,
\end{align}
where $0<\sigma<1<r \leq 2^* -1$. The systematic study of such problems was initiated by the celebrated works of Brezis-Nirenberg \cite{brezis1983existence} and Ambrosetti-Brezis-Cerami \cite{ambrosetti1994convexconcave}, which laid the foundation for the analysis of existence, non-existence, and multiplicity of positive solutions for variational elliptic problems. The theory was subsequently extended to quasilinear equations driven by $p$-Laplacian; see, for instance, \cite{ambrosettigarciaperal1996multiplicity}. Multiplicity results in the critical growth regime were later obtained in \cite{garcia1994critical, garcia2000minimizer}, under additional restrictions on the range of $p$. In recent years, analogous convex–concave type problems have also been investigated for nonlocal and mixed local–nonlocal operators, leading to various existence and multiplicity results; for which we refer to \cite{brandle2013fractional, barrios2015convexconcave, weisu2015fractional, yezhang2024multiplicity, silva2024mixed, biagi2025convexconcave, dhanya2025multiplicityresultsmixedlocal, bhakta2025quasilinearproblemsmixedlocalnonlocal} among others.

A parallel line of research replaces the concave term with a singular nonlinearity, leading to problems of the form
\begin{align}\label{Lu_eqn}
    \mathcal{L} u = \frac{\la}{u^{\delta}} + u^r \text{ in } \Om,\ u>0 \text{ in } \Om,\ u=0 \text{ on } \partial \Omega
\end{align}
where $0<\delta<1$. Despite the singularity, these problems exhibit bifurcation and multiplicity phenomena similar to those observed in convex–concave equations. For the classical Laplacian $\mathcal{L}=-\Delta$, early works such as \cite{coclite1989singular, yijing2001subcritical} and \cite{hirano2004critical} established multiplicity results for small values of $\la$. Later, Haitao \cite{haitao2003multiplicity} provided a global multiplicity result for \eqref{Lu_eqn} in the Laplacian setting, and Giacomoni-Schindler-Tak\'a\v{c} \cite{giacomoni2007sobolev} extended these findings to the $p$-Laplacian. A notable feature in both \cite{haitao2003multiplicity} and \cite{giacomoni2007sobolev} is that the first solution is obtained as a local minimizer of the associated energy functional, without any smallness assumption on $\la$. Recently, attention has also turned to nonlocal operators, such as $\mathcal{L}=\s$ or $\mathcal{L}=\ps$. For related existence and multiplicity results, see \cite{barrios2015multiplicity, giacomoni2019verysingular, adimurthi2018fractional, tuhina2019nehari} and the references therein.

Inspired by the developments in the purely local and purely nonlocal settings, recent studies address the PDE \eqref{Lu_eqn} in the context of mixed local-nonlocal operators. In particular, \cite{garain2023subcritical} and \cite{bal2024multiplicitysolutionsmixedlocalnonlocal} examined the subcritical problems associated with $\mathcal{L}= -\De + \s$ and $\mathcal{L}= -\De_p + \ps$ respectively and obtained two positive solutions for small $\la$, leading to local multiplicity results. In the critical case, Biagi et. al. \cite{biagi2024critical} considered the operator $\mathcal{L}= -\De + \epsilon \s$ with $r= 2^* -1$ and established multiplicity of solutions for sufficiently small values of both $\epsilon$ and $\la$. Following the approaches of \cite{ biagi2024critical, biagi2025convexconcave}, further studies focused on the operator $\mathcal{L}=-\De_p +\epsilon \ps$, particularly in the critical growth regime; see, for instance, \cite{bhakta2025quasilinearproblemsmixedlocalnonlocal, sanjit2025multiplicity}. Although the parameter $\epsilon$ serves as a convenient technical tool in establishing the existence of multiple solutions, the resulting multiplicity conclusions in these articles are essentially local in nature. 

This limitation motivates the present work. We develop a global bifurcation-type framework that yields existence, non-existence, and multiplicity results in both subcritical and critical regimes for mixed nonlinear operators with singular nonlinearities. A sharp threshold value $\Lambda$ separating existence and non-existence regions with respect to $\lambda$ is determined via non-existence results for a generalized eigenvalue problem associated with mixed $(p,q)$ operators.
We then establish a global multiplicity result in the subcritical case, which is new even in the linear case. The main contribution of this article lies in obtaining the global multiplicity for the critical exponent problem in the presence of a singular nonlinearity. This is achieved through refined estimates of Aubin-Talenti-type functions and is new even in the homogeneous case $p=q$. Two auxiliary results play a central role in our analysis and are of independent interest: a Hopf-type Strong Comparison Principle and a Sobolev versus H\"older local minimizer result for mixed $(p,q)$ operators. These results are instrumental in lifting multiplicity from a local to a global setting and in recovering the mountain-pass geometry of the associated energy functional, particularly in the nonlinear case.

With this conceptual overview of our main results, we now turn to presenting the precise statements of the theorems. To this end, we first introduce the energy space and formalize the notion of weak solutions to \eqref{P_lambda}. The natural solution space for  $(P_\lambda)$ is $\mathbb{X}_0 (\Om)$, defined as  $$\mathbb{X}_0 (\Om):= \wop \cap W^{s,q}_0 (\Om),$$ which endowed with the norm $\|u \|_{\mathbb{X}_0 (\Om)}:= \|u\|_{\wop} + \|u\|_{W^{s,q}_0 (\Om)}$, is a Banach space. When $q \leq p$, the chain of embeddings $\wop \hookrightarrow W^{1,q}_0(\Om) \hookrightarrow W^{s,q}_0 (\Om)$ ensures that $\mathbb{X}_0 (\Om) = \wop.$  On the other hand, when $p<q$, the absence of a suitable embedding forces us to work with the intersection of two Sobolev spaces $\mathbb{X}_0 (\Om)$ as it stands.

\begin{definition}
    We say $u\in \mathbb{X}_0(\Om)$ is a weak solution of \eqref{P_lambda} if $\text{ess}\inf_{\omega}  u >0$ for every $\omega\Subset\Om$ and for every $\varphi\in\test$, there holds
    \begin{multline*}
        \int_{\Om} |\na u|^{p-2} \na u \cdot \na \varphi\ dx + \int_{\sr^N \times \sr^N} \frac{|u(x) - u(y)|^{q-2} (u(x) - u(y)) (\varphi(x) - \varphi(y))}{|x-y|^{N+sq}}\ dx\ dy\\
        = \int_{\Om} \frac{\la}{u^{\delta}} \varphi\ dx + \int_{\Om} u^r \varphi\ dx.
    \end{multline*}
\end{definition}

The first main result addresses the existence and non-existence of solutions for \eqref{P_lambda}.

\begin{theorem}[Existence and non-existence]\label{Theorem_first_solution}
    There exists $\Lambda > 0$ such that the problem \eqref{P_lambda} admits
    \begin{itemize}
        \item[(a)] at least one solution for every $0 < \la \leq \Lambda$,
        \item[(b)] a minimal solution for every $0 < \la \leq \Lambda$ and the minimal solutions are strictly increasing with respect to $\la$,
        \item[(c)] no solution for every $\la > \Lambda$.
    \end{itemize}
\end{theorem}

Next, we move on to the global multiplicity results and present the subcritical case first.

\begin{theorem}[Global multiplicity in the subcritical case]\label{Theorem_second_solution_subcritical}
    Suppose $r< r^*$ and $p,q,s$ satisfy one of the following assumptions:
    \begin{itemize}
        \item[(i)] (linear case) $p=q=2$ and $s\in (0,1)$,
        \item[(ii)] (nonlinear case) $p\in (1,\infty)$, $q \geq 2$ and $s\in \left(0, \frac{1}{q'} \right)$, $q'$ being the H\"older conjugate of $q$.
    \end{itemize}
    Then the problem \eqref{P_lambda} admits at least two distinct solutions for every $\la\in (0, \Lambda)$.
\end{theorem}

Turning to the critical case, the classical variational approach based on Ekeland’s principle and perturbations via Aubin–Talenti functions cannot be directly applied in the mixed local–nonlocal setting. This is due to the lack of attainment of the best constant in the mixed Sobolev inequality(see \cite{biagi2025brezisnirenberg}) and the loss of scale invariance of the operator $-\p + \ps.$ Consequently, careful estimates yield structural conditions on 
$p,q,s$ and $N$ ensuring a global multiplicity.

\begin{theorem}[Global multiplicity in the critical case]\label{Theorem_second_solution_critical}
    Suppose that $2\leq q\leq p$ with $p \in \left( \frac{3N}{N+3},\ 3 \right)$ and $r= r^* = p^* -1$. Additionally, let
    \[0<s< \min \left\{ \frac{1}{q'},\ 1- \frac{1}{q} \left(\frac{N-p}{p-1} - N \left( 1- \frac{q}{p} \right) \right) \right\},\]
    where $q'$ is the H\"older conjugate of $q$. Then the problem \eqref{P_lambda} admits at least two distinct solutions for every $\la\in (0, \Lambda)$.
\end{theorem}

The above result may be viewed in relation to approaches involving the $\epsilon$-dependent mixed operators, the discussion on which was postponed earlier. In such works, a sufficiently small choice of $\epsilon$ moderates the impact of the nonlocal term, thereby creating a favorable variational structure to admit a second solution when $\epsilon$ and $\lambda$ both are small. More precisely, the $\epsilon$-dependent framework generally ensures the existence of a first positive solution for all $\la\in (0,\Lambda_\epsilon)$, while obtaining a second solution is restricted solely to a smaller subinterval $(0,\Lambda_*)$ with $\Lambda_* < \Lambda_{\epsilon}$. As a result, the intermediate zone $(\Lambda_*, \Lambda_\epsilon)$ remains unresolved, leaving the overall multiplicity picture incomplete. As for global multiplicity in the case $\epsilon=1$, it was obtained when $p=q=2$, $s\in \left(0, \frac{1}{2}\right)$, $N=3$ and $r=2^* -1$ (see corollary 1.3, \cite{biagi2024critical} and corollary 1.5, \cite{biagi2025convexconcave}). Notably, when $p=q=2$, the range obtained in Theorem \ref{Theorem_second_solution_critical} coincides with the one obtained in \cite{biagi2024critical} and also the range of $s$ complements the same contained in \cite{dhanya2025multiplicityresultsmixedlocal}.

We conclude the discussion on global multiplicity by emphasizing that our results rely on a framework that extends existing work on mixed operators and captures genuinely global effects for nonlinear problems. This approach is driven by two key tools: a Hopf-type Strong Comparison Principle and an equivalence between Sobolev and H\"older local minimizers, both new in this setting. We now turn to these results. We now turn to a detailed analysis of these two tools.

\subsection{Strong Comparison Principle}

The Strong Comparison Principle (SCP) is a fundamental concept that extends the classical Strong Maximum Principle (SMP) and Hopf's lemma to a comparison framework. Here, we establish SCP and Hopf-type SCP for mixed operators in the presence of singular nonlinearities.

\begin{theorem}[Comparison Principle]\label{Theorem_SCP}
    Let $1<p,q< \infty$ with $p>sq$, $0< \delta <1$ and $f_1,f_2$ be non-negative continuous functions in $\Om$ satisfying $f_1 \leq f_2$, $f_1 \not\equiv f_2$ in $\Om$. Assume, in addition, one of the following holds:
    \begin{itemize}
        \item[(i)] (linear case) $p=q=2$ and $s\in (0,1)$,
        \item[(ii)] (nonlinear case) $p\in (1,\infty)$, $q \geq 2$ and $s\in \left(0, \frac{1}{q'} \right)$, $q'$ being the H\"older conjugate of $q$.
    \end{itemize}
    If $u_1, u_2 \in C^{1,\alpha}(\overline{\Om})$, for some $\alpha\in(0,1)$, $u_1, u_2 >0$ in $\Om$ and they satisfy
    \begin{equation}\tag{$Q$}\label{scp equations}
        \begin{aligned}
            -\p u_1 + \qs u_1 &- \frac{1}{u_1^{\delta}} = f_1 \text{ in } \Om\\
            -\p u_2 + \qs u_2 &- \frac{1}{u_2^{\delta}} = f_2 \text{ in } \Om\\
            u_1 =u_2 &=0 \text{ in } \bdry
        \end{aligned}
    \end{equation}
    then the following hold:
    \begin{itemize}
        \item[(a)] Strong Comparison Principle: $u_1 < u_2$ in $\Om$,
        \item[(b)] Hopf-type Strong Comparison Principle: $\frac{\pa u_2}{\pa\nu} < \frac{\pa u_1}{\pa\nu}$ on $\pa\Om$,
    \end{itemize}
    where $\nu = \nu (x_0)$ is the outward unit normal to $\pa\Om$ at $x_0 \in \pa\Om$.
\end{theorem}

For local elliptic operators, both the SCP and Hopf-type SCP are well studied in the literature. In contrast, corresponding results for nonlocal and mixed local–nonlocal operators are more limited, particularly in the presence of singular nonlinearities. We therefore briefly review existing results on the SCP and Hopf-type SCP.

\begin{itemize}
    \item \textbf{Local operators.} In the case of the Laplacian,\cite{dhanya2015three} established the SCP and then the Hopf-type SCP by deriving a pointwise lower bound for $u_2-u_1$ in terms of the first eigenfunction of the Laplacian and a suitable power of it. For the $p$-Laplace operator, \cite{giacomoni2007sobolev} proved the SCP and its Hopf-type version under a stricter condition $f_1 < f_2$ in $\Om.$ Later, \cite{indulekha2023scp} showed that for $p>2$, relaxing this to $f_1 \leq f_2$, $f_1 \not\equiv f_2$, may lead to the failure of the SCP, indicating the degeneracy of the $p$-Laplacian for $p>2$ and the need for a strict inequality.

    \item \textbf{Nonlocal operators.} For nonlocal operators, weaker regularity necessitates a Hopf-type SCP formulated through boundary distance estimates of the form $u_2-u_1 \geq c d(x)^s$ in $\Om$, where $d(x)$ is the distance function and $c>0$. For the fractional Laplacian, although not stated explicitly, Theorem 4.7 of \cite{giacomoni2019existence} effectively proves both the SCP and the fractional Hopf-type SCP. However, to the best of our knowledge, the corresponding nonlinear result for the fractional $p$-Laplacian remains open, where one expects a similar boundary estimate.

    \item \textbf{Mixed local-nonlocal operators.} In this setting, \cite{dhanya2026regularity} extended the ideas of \cite{jarohs2018strong} to prove the SCP for singular problems, however, under a technical restriction $s\in \left(0, \frac{1}{q'} \right)$, $q'$ being the H\"older conjugate of $q$. For the mixed linear operator $\mathcal{L} = -\De + (-\De)^s$, \cite{pramanik2024mixed} employed a weak Harnack-type inequality followed by a construction-based approach and established both SCP and Hopf-type SCP for the full range $s\in(0,1)$. Nevertheless, the method strongly relies on the linearity of the operator, thus lacking a direct extension for nonlinear operators.
\end{itemize}

In summary, for mixed operators with singular nonlinearity, the SCP is known in the nonlinear case only when $s\in \left(0, \frac{1}{q'} \right)$, while the Hopf-type SCP is available solely for linear operators. This gap highlights both the technical difficulties and the novelty of our result. Our proof stems from the observation that a positive solution to a carefully designed infinite semipositone problem serves as a subsolution to a linearized version of the original problem. The resulting divergence form operator involves only a H\"older continuous coefficient matrix which lacks an explicit expression, rendering standard sub-supersolution methods ineffective. Nevertheless, invoking the implicit function theorem, we obtain a positive solution to the associated semipositone problem involving the linearized local operator and combine it with refined estimates on the nonlocal term to close the argument.

\subsection{Sobolev versus H\"older local minimizers}

Next, we focus on the second key ingredient towards the multiplicity result, namely a Sobolev versus H\"older local minimizer result tailored to mixed local-nonlocal operators. In the seminal work of Brezis-Nirenberg \cite{brezis1993minimizer}, they considered critical growth functionals $\Phi: H^1_0 (\Om) \rar \sr$ satisfying a certain structure and established that if $u_0$ is a local minimizer of $\Phi$ in $C^1_0$-topology, then $u_0$ remains a local minimizer of $\Phi$ in the $H^1_0$-topology as well. This result was subsequently extended to the quasilinear settings by \cite{garcia2000minimizer, guo2003minimizer} and \cite{brock2008multiplicity}, and later generalized to include singular nonlinearities by Giacomoni et al. \cite{giacomoni2007sobolev}. In the nonlocal framework, Iannizzotto and his collaborators established the analogous results, see \cite{iannizzotto2015h, iannizzotto2020minimizer} for instance.

In this paper, we extend this line of research to a mixed local-nonlocal setting in the presence of a singular term. We consider the functional $I: \xo(\Om) \rar \sr$ given by
\begin{equation*}
    I(u):=\frac{1}{p} \int_{\Om} |\na u|^p\ dx + \frac{1}{q} \int_{\sr^N \times \sr^N} \frac{|u(x)- u(y)|^q}{|x-y|^{N+qs}}\ dx\ dy - \frac{1}{1-\delta} \int_{\Om} \left(u^+ \right)^{1-\delta} dx - \int_{\Om} F(x, u^+)\ dx,
\end{equation*}
where $0<\delta <1$, $\displaystyle u^+ := \max \{u, 0\}$ and $\displaystyle F(x,t):= \int_0^t f(x,\tau)\ d\tau$. The prototype of the function $f$ is $f(x,t)= t^r,\ t\geq 0$, where $r$ is a positive number between $\max \{ p-1, q-1\} < r \leq \max \{p^* -1, q_s^* -1 \}$. The precise assumptions on $f$ are outlined in Section \ref{section_minimizer}, under which in force, we prove the following theorem.

\begin{theorem}\label{Theorem_minimizer}
    Let $1<p,q<\infty$, $s\in (0,1)$ with $p>sq$ and $0<\delta<1$. Suppose $f$ satisfies \textit{(f1)}, \textit{(f2)} and $I: \xo(\Om) \rar \sr$ is defined as above. Let $u_0 \in C^1_0 \omc$ satisfy
    \[u_0 (x) \geq k_0 d(x) \text{ in } \Om, \text{ for some } k_0 >0.\]
    If $u_0$ is a local minimizer of $I$ in $C^1_0 \omc$-topology, then $u_0$ is a local minimizer of $I$ in $\xo (\Om)$-topology as well.
\end{theorem}

We prove the result via a constrained minimization approach, inspired by \cite{brock2008multiplicity}, combined with uniform $C^{1,\alpha}$-estimates. A key step in the process is to derive a uniform $L^{\infty}(\Om)$-bound, which is essential for subsequently obtaining the uniform $C^{1,\alpha}$-bound. In a recent work \cite{dhanya2026regularity} (Theorem 15), such an $L^{\infty}$-estimate is derived in the subcritical case $r < \min \{ p^* -1, q_s^* -1 \}$ through a De Giorgi-Stampacchia iteration. While this approach extends naturally to the wider range $r < \max \{ p^* -1, q_s^* -1 \}$, we refine it further in lemma \ref{l_infinity} with the help of Moser iteration method and include the critical case $r= \max \{ p^* -1, q_s^* -1 \}$. This leads to a uniform $L^{\infty}(\Om)$-bound valid over an expanded range of $r$, thereby extending the current literature.

Regarding the structure of the paper, in Section \ref{section_Hopf_type_SCP}, we establish the SCP and the Hopf-type SCP for mixed operators. Section \ref{section_minimizer} addresses the Sobolev versus H\"older local minimizer result of Theorem \ref{Theorem_minimizer}. Finally, Section \ref{section_first_solution} contains the proof of Theorem \ref{Theorem_first_solution}, establishing the existence and non-existence of solutions, and Section \ref{section_second_solution} concludes with the multiplicity results stated in Theorem \ref{Theorem_second_solution_subcritical} and Theorem \ref{Theorem_second_solution_critical}. Lastly, we use the following notations throughout the paper.
\vspace{5pt}

\textbf{Notations.}

\begin{itemize}

    \item We denote by $\Om$ a smooth, bounded domain in $\sr^N$ and $\nu = \nu (x_0)$ denotes the outward unit normal to $\pa\Om$ at $x_0 \in \pa\Om$.
    
    \item For any point $x\in\Om$, $d(x)$ denotes the distance of the point $x$ from $\pa\Om$, defined as
    \[d(x)= dist (x, \pa\Om) = \inf_{y\in\pa\Om}\ |y-x|.\]

    \item For $\eta>0$, we define an $\eta$-neighbourhood of the boundary $\pa\Om$ as
    \[\Om_{\eta} = \left\{ x\in \Om: d(x) < \eta \right\}.\]

    \item For two functions $f_1, f_2 \in C\omc$, we write $f_1\sim f_2$ to express that there exist $c_1, c_2 >0$ such that $c_1 f_2(x) \leq f_1(x) \leq c_2 f_2(x)$ for all $x\in\Om$.

    \item We also adopt the notation $\mathcal{G}u(x,y) = |u(x)-u(y)|^{q-2} \left( u(x) - u(y) \right)$.
\end{itemize}


\section{Comparison principle}\label{section_Hopf_type_SCP}

In this section, we establish the comparison principle stated in Theorem \ref{Theorem_SCP}. We begin with part $(a)$ of Theorem \ref{Theorem_SCP}, that is, the Strong Comparison Principle.

\begin{theorem}[Strong Comparison Principle]\label{Theorem_only_scp}
    Under the assumptions of Theorem \ref{Theorem_SCP}, we have $u_1<u_2$ in $\Om$.
\end{theorem}

\begin{proof}
    In the linear case, that is, when $p=q=2$ and $s\in (0,1)$, the result follows from Theorem 3 of \cite{pramanik2024mixed}. In the nonlinear setting $\left( p>1,\ q\geq 2,\ s\in \left( 0, \frac{1}{q'} \right) \right)$, the conclusion is a consequence of Theorem 5, \cite{dhanya2026regularity}. Although in \cite{dhanya2026regularity}, the result is stated under a stricter assumption $f_1 < f_2$ in $\Om$, a careful inspection of the proof shows that the argument remains valid even under the weaker condition $f_1 \leq f_2$ in $\Om$.
\end{proof}

Next, we turn to part $(b)$ of Theorem \ref{Theorem_SCP}. Before proceeding to the proof, we state two auxiliary results. The first concerns the existence of a positive solution to an infinite semipositone problem. The second one deals with the extension of a H\"older continuous function defined on a bounded domain to the whole space $\sr^N$, in such a way that the extended function possesses certain desirable regularity and sign properties. Although a result similar to the latter one is proved in lemma 6.2 of \cite{pramanik2024mixed} for a specific function, the argument naturally extends to a broader class of functions, as presented in lemma \ref{C_1_alpha extension} below.

\begin{proposition}\label{infinite_semipositone_existence}
    Let $U$ be a smooth, bounded domain in $\sr^N$ and consider the following problem
    \begin{equation}\tag{$Q_{\theta}$}\label{Q theta}
    \begin{aligned}
        -div \left(A_0 (x) \na z \right) &= z^{\si} - \frac{\theta}{z^{\gamma}} \text{ in } U\\
        z > 0 \text{ in } U,\ z &= 0 \text{ on } \pa U
    \end{aligned}
    \end{equation}
    where $\sigma, \gamma \in (0,1)$ and $\theta>0.$ We assume $A_0 (x) =\left( a_{ij}^{(0)} (x) \right)_{N\times N}$ is a symmetric and positive definite matrix on $U$ with coefficients $a_{ij}^{(0)}$ $(i,j = 1,2, \dots, N)$ H\"older continuous on $\overline{U}$. Then, for $\theta$ sufficiently small, \eqref{Q theta} admits a positive solution $z_{\theta} \in \coa$, for some $\alpha\in (0,1)$. Moreover, $z_{\theta} \geq c_{\theta} d(x)$, for some $c_{\theta} >0$.
\end{proposition}

We defer the proof of this lemma to the appendix (see appendix \ref{appendix_infinite_semipositone}). For the next auxiliary result, we give a brief sketch of the proof.

\begin{lemma}\label{C_1_alpha extension}
    Let $U$ be a smooth, bounded domain in $\sr^N$ and $z\in C^{1,\alpha}(\overline{U})$ satisfy $z>0$ in $U$, $z=0$ on $\pa U$ and $\frac{\pa z}{\pa \nu} < 0$ on $\pa U$. Then there exists an extension $\tilde{z}$ of $z$ such that $\tilde{z} \in C^{1, \alpha} (\sr^N)$, $\tilde{z} =z$ in $\overline{U}$ and $\tilde{z} \leq 0$ in $\sr^N \setminus U$.
\end{lemma}

\begin{proof}
    We prove this result by the standard method of flattening the boundary. First, we note that any given $C^{1,\alpha}$ function $v$ in the upper half space vanishing on the boundary can be extended as a $C^{1,\alpha}(\mathbb{R}^N)$ function using an odd reflection. For a given function $z$ defined in a bounded domain $U$, we utilize the $C^{1,\alpha}$ regularity of local coordinate charts along with a partition of unity argument to construct a $C^{1,\alpha}$ extension $\tilde{z}$ in $\mathbb{R}^N$. This extension $\tilde{z}$ also satisfies $\tilde{z} \leq 0$ in $\mathbb{R}^N \setminus U$ and ensures all the required properties.
\end{proof}

Now, we proceed to the proof of Theorem \ref{Theorem_SCP}.
\vspace{5pt}

\noindent\textbf{\textit{Proof of Theorem \ref{Theorem_SCP}:}} Set \( w := u_2 - u_1 \). Then part $(a)$ of Theorem \ref{Theorem_SCP}, namely the Strong Comparison Principle, follows directly from Theorem \ref{Theorem_only_scp}. Therefore, it remains to establish part $(b)$ of Theorem \ref{Theorem_SCP}, which corresponds to the Hopf-type Strong Comparison Principle. For clarity and completeness, we divide the proof of the Hopf-type SCP into three cases, according to the values of $p$ and $q$.
\vspace{5pt}

\noindent\textbf{\textit{Case I: Linear case ($p=q=2$).}} The desired result in this case has been proved in Theorem 3 of \cite{pramanik2024mixed} for all $s\in (0,1)$.
\vspace{5pt}

\noindent\textit{\textbf{Case II: Nonlinear case with $p>1$, $q=2$ and $s\in \left( 0, \frac{1}{2} \right)$.}} In this case, the nonlocal component is linear. Since $u_1, u_2$ satisfy \eqref{scp equations}, using the mean value theorem and by linearizing the operator, we obtain that $w$ solves the linear equation
\begin{equation}\label{w_equn_1}
    \begin{aligned}
        -div \left(A(x)\na w \right) + \s w + \frac{\delta}{\xi (x)^{\delta +1}} w = f_2 - f_1 \geq 0 \text{ in } \Om,
    \end{aligned}
\end{equation}
for some function $\xi \in [u_1, u_2]$. The coefficient matrix $A(x) = \left(a_{ij} (x)\right)_{N\times N}$ is given by
\[ a_{ij} (x) = \int_0^1 |\na u_t (x)|^{p-2} \left[ \delta_{ij} + (p-2) \frac{\pa u_t}{\pa x_i} \frac{\pa u_t}{\pa x_j} \frac{1}{|u_t (x)|^2} \right] dt,\ x\in \Om, \]
where $u_t := (1-t) u_1 + t u_2$ and $\delta_{ij}$, $i,j = 1,2, \dots, N$, is the standard Kronecker delta symbol. The derivation of the expression for $a_{ij}(x)$ is standard and can be found in Theorem 2.3 of \cite{giacomoni2007sobolev}.  {Since the outward normal derivatives $\frac{\pa u_1}{\pa\nu}$ and $ \frac{\pa u_2}{\pa\nu}$ are strictly negative on $\pa\Om$ (see Theorem 1.2, \cite{antonini2023global}),} the operator $\mathcal{L}_1 z:= div \left(A(x)\na z \right)$ is uniformly elliptic in $\Om_{\eta} = \left\{ x\in \Om: d(x) < \eta \right\}$, for $\eta >0$ sufficiently small. We fix one such  $\eta$ for the remaining part of the proof. Moreover, since $u_1 (x) \sim d(x)$ in $\Om$, from \eqref{w_equn_1} it follows that for some constant $k>0,$
\begin{equation}\label{w_equn_2}
    -div \left(A(x)\na w \right) + \s w + \frac{k}{d(x)^{\delta +1}} w \geq 0 \text{ in } \Om_{\eta}.
\end{equation}

We now aim to find a function that is positive in $\Om$, exhibits a distance-like behavior near $\pa \Omega$ and satisfies \eqref{w_equn_2} with a reverse inequality. More precisely, we claim the following.

\noindent\textit{\textbf{Claim 1:}} There exist $\epsilon >0$ and a function $\tilde{w} \in C^{1,\alpha} (\sr^N)$ such that $\tilde{w} >0$ in $\Om_{\eta}$, $\tilde{w}=0$ on $\pa\Om$, $\tilde{w} \leq 0$ in $\sr^N \setminus \Om_{\eta}$, $\frac{\pa \tilde{w}}{\pa\nu} < 0$ on $\pa\Om_{\eta}$ and satisfies
\begin{equation}
    \begin{aligned}
        -div \left(A(x)\na \left(\epsilon \tilde{w} \right) \right) + \s \left(\epsilon \tilde{w} \right) + \frac{k}{d(x)^{\delta +1}} \left(\epsilon \tilde{w} \right) \leq 0 \text{ in } \Om_{\eta_0}\\
        \epsilon \tilde{w} \leq w \text{ in } \sr^N \setminus \Om_{\eta_0}
    \end{aligned}
\end{equation}
for some $\eta_0 \in (0, \eta)$.
\vspace{5pt}

To prove this claim, we consider the following infinite semipositone problem:
\begin{equation}\label{v_equn}
    \begin{aligned}
        -div \left(A (x) \na z \right) = z^{\si} &- \frac{\theta}{z^{\gamma}} \text{ in } \Om_{\eta}\\
        z > 0\text{ on } \Om_\eta,\ z &= 0 \text{ on } \pa\Om_{\eta}
    \end{aligned}
\end{equation}
where $0< \si <1$, $\delta < \gamma <1$ and $\theta>0$ is small.
By Proposition \ref{infinite_semipositone_existence}, there exists 
\( z_\theta \in C^{1,\alpha}(\overline{\Omega_\eta}) \) that solves \eqref{v_equn} and satisfies the assumptions of lemma \ref{C_1_alpha extension}.
Hence, applying lemma \ref{C_1_alpha extension}, we can extend \( z_\theta \) to a function \( \tilde{w} \in C^{1,\alpha}(\mathbb{R}^N) \) with $\tilde{w}\leq 0$ in $\mathbb{R}^N\setminus \Om_\eta$. Using the properties of $\tilde{w}$, we have,
\begin{equation}\label{estimate 1}
    -div \left(A(x) \na \tilde{w} \right) + \frac{k}{d(x)^{\delta +1}} \tilde{w} = \tilde{w}^{\si} - \frac{\theta}{\tilde{w}^{\gamma}} + \frac{k \tilde{w}}{d(x)^{\delta +1}} \leq c_1 d(x)^{\si} - \frac{c_2}{d(x)^{\gamma}} + \frac{c_3}{d(x)^{\delta}} \text{ in } \Om_{\eta},
\end{equation}
for some constants $c_1, c_2, c_3 >0$. Now, since $\tilde{w} \in C^{1, \alpha} (\sr^N)$ and $s\in \left( 0, \frac{1}{2} \right) $, by Proposition 2.6 of Silvestre \cite{silvestre2007regularity}, there exists $c_4 >0$ such that $\left|\s \tilde{w}\right| \leq c_4$ in $\Om_{\eta}$. Thus, as $\gamma >\delta$, choosing $\eta_0 \in (0, \eta )$ small enough, from \eqref{estimate 1} we obtain
\begin{equation}\label{v_tilde equn}
    -div \left(A(x) \na \tilde{w} \right) + \s \tilde{w} + \frac{k}{d(x)^{\delta +1}} \tilde{w} \leq 0 \text{ in } \Om_{\eta_0}.
\end{equation}
Next, as $w>0$ in $\Om$ by Theorem \ref{Theorem_only_scp}, we choose $\epsilon >0$ such that $w \geq \epsilon \tilde{w}$ in the annulus $\left\{x\in \Om: \eta_0 \leq d(x) \leq \eta \right\}.$ Also, by construction, $\tilde{w}\leq 0$ in $\sr^N \setminus \Om_\eta$. Therefore, $w \geq \epsilon \tilde{w}$ in $\sr^N \setminus \Om_{\eta_0}$ and this proves our claim.

Combining \eqref{w_equn_2} and \eqref{v_tilde equn}, we thus have
\begin{equation}\label{comparison q=2 case}
    \begin{aligned}
        -div \left(A(x)\na w \right) + \s w + \frac{k}{d(x)^{\delta +1}} w \geq 0 \text{ in } \Om_{\eta_0}\\
        -div \left(A(x)\na \left(\epsilon \tilde{w} \right) \right) + \s \left(\epsilon \tilde{w} \right) + \frac{k}{d(x)^{\delta +1}} \left(\epsilon \tilde{w}\right) \leq 0 \text{ in } \Om_{\eta_0}\\
        w \geq \epsilon \tilde{w} \text{ in } \sr^N \setminus \Om_{\eta_0}.
    \end{aligned}
\end{equation}
By the weak comparison principle, $w \geq \epsilon \tilde{w}$ in $\Om_{\eta_0}$. Consequently, as $w= \tilde{w} =0$ on $\pa\Om$, it follows that $\frac{\pa w}{\pa \nu} \leq \frac{\pa \left(\epsilon \tilde{w} \right)}{\pa \nu} < 0$ on $\pa\Om$, or equivalently, $\frac{\pa u_2}{\pa \nu} < \frac{\pa u_1}{\pa \nu}$ on $\pa\Om$.
\vspace{5pt}

\noindent\textit{\textbf{Case III: Nonlinear case with $p>1$, $q>2$ and $s\in \left( 0, \frac{1}{q'} \right)$.}} We proceed as in case II except for the nonlocal term, which is now nonlinear in nature. Writing $w=u_2-u_1$ and considering $\Om_{\eta}$ the same as in case II, here we obtain the following:
\begin{equation}\label{w_equn_general case}
    \begin{aligned}
        -div \left(A(x)\na w \right) + \mathcal{L}_2 w &+ \frac{k}{d(x)^{\delta +1}} w \geq 0 \text{ in } \Om_{\eta}\\
        w >0 \text{ in } \Om_{\eta},\ w &\geq 0 \text{ in } \sr^N \setminus \Om_{\eta}.
    \end{aligned}
\end{equation}
Here, $\displaystyle{\mathcal{L}_2 w (x):= \int_{\sr^N} a(x,y)\ \frac{w(x) - w(y)}{|x-y|^{N+2s}}\ dy}$ and with $u_t = (1-t) u_1 + t u_2$, the weight function $a(x,y)$ is given by
\[a(x,y):= (q-1) \frac{1}{|x-y|^{(q-2)s}} \int_0^1 |u_t (x) - u_t (y)|^{q-2}\ dt.\]

Comparing \eqref{w_equn_general case} with \eqref{w_equn_2}, it is evident that the only modification arises in the nonlocal component. Thus, the arguments used in case II can be applied without alteration, provided that we prove the following claim.
\vspace{5pt}

\noindent\textbf{\textit{Claim 2:}} There exists $c_5 >0$ such that $\left|\mathcal{L}_2 \tilde{w} \right|< c_5$ in $\Om_{\eta}$, where $\tilde{w}$ is constructed similarly as in case II.
\vspace{5pt}

For every $x\in \overline{\Om_{\eta}}$, we have
\begin{align}\label{I_1 + I_2}
    \left|\mathcal{L}_2 \tilde{w} (x) \right| &\leq \int_{B_1 (x)} |a(x,y)|\ \frac{|\tilde{w} (x) - \tilde{w} (y)|}{|x-y|^{N+2s}}\ dy + \int_{B_1^c (x)} |a(x,y)|\ \frac{|\tilde{w} (x) - \tilde{w} (y)|}{|x-y|^{N+2s}}\ dy \nonumber\\
    &=: J_1 + J_2 \text{ (say)}.
\end{align}
We estimate $J_1$ and $J_2$ separately. Since $u_1, u_2, \tilde{w}$ are Lipschitz in $\sr^N$ and $0< s < \frac{1}{q'}$, for $J_1$ we have:
\begin{align}\label{estimate for I_1}
    J_1 &= \int_{B_1 (x)} |a(x,y)|\ \frac{|\tilde{w} (x) - \tilde{w} (y)|}{|x-y|^{N+2s}}\ dy\nonumber\\
    &= \int_{B_1 (x)} (q-1) \left(\int_0^1 |u_t (x) - u_t (y)|^{q-2}\ dt\right) \frac{|\tilde{w} (x) - \tilde{w} (y)|}{|x-y|^{N+qs}}\ dy\nonumber\\
    &\leq c_6 \int_{B_1 (x)} |x-y|^{q(1-s) -1-N} dy = \frac{c_7}{q(1-s)-1},
\end{align}
where $c_7$ is independent of $x$. Now, we use the fact $q \geq 2$ and estimate $J_2$ as follows:
\begin{align}\label{estimate for I_2}
    J_2 &= \int_{B_1^c (x)} |a(x,y)|\ \frac{|\tilde{w} (x) - \tilde{w} (y)|}{|x-y|^{N+2s}}\ dy\nonumber\\
    &= \int_{B_1^c (x)} (q-1) \left( \int_0^1|u_t (x) - u_t (y)|^{q-2} dt \right) \frac{|\tilde{w} (x) - \tilde{w} (y)|}{|x-y|^{N+qs}}\ dy\nonumber\\
    &\leq c_8 \int_{B_1^c (x)} |x-y|^{-N -qs}\ dy = \frac{c_9}{qs},
\end{align}
where $c_9$ is again independent of $x$. Therefore, \eqref{I_1 + I_2}, \eqref{estimate for I_1} and \eqref{estimate for I_2} together establish claim 2.

Consequently, proceeding similar to case II, analogous to \eqref{comparison q=2 case}, here we have
\begin{equation}\label{comparison general case}
    \begin{aligned}
        -div \left(A(x)\na w \right) + \mathcal{L}_2 w + \frac{k}{d(x)^{\delta +1}} w \geq 0 \text{ in } \Om_{\eta_0}\\
        -div \left(A(x)\na \left(\epsilon \tilde{w} \right) \right) + \mathcal{L}_2 \left(\epsilon \tilde{w} \right) + \frac{k}{d(x)^{\delta +1}} \left(\epsilon \tilde{w} \right) \leq 0 \text{ in } \Om_{\eta_0}\\
        w \geq \epsilon \tilde{w} \text{ in } \sr^N \setminus \Om_{\eta_0}
    \end{aligned}
\end{equation}
for some $\epsilon>0$ and $\eta_0 \in (0, \eta)$ sufficiently small. Hence, by the weak comparison principle, $w \geq \epsilon \tilde{w}$ in $\Om_{\eta_0}$. Since $w= \tilde{w}=0$ on $\pa\Om$, we obtain $\frac{\pa u_2}{\pa \nu} < \frac{\pa u_1}{\pa \nu}$ on $\pa\Om$ in this case as well.

Therefore, the above three cases considered together complete the proof of part $(b)$, i.e., the Hopf-type Strong Comparison Principle, of Theorem \ref{Theorem_SCP}.\hfill\qed


\section{Sobolev versus H\"older local minimizers}\label{section_minimizer}

This section is dedicated to the proof of Theorem \ref{Theorem_minimizer}. We consider the energy functional $I: \xo(\Om) \rar \sr$, given by
\begin{equation*}
    I(u):=\frac{1}{p} \int_{\Om} |\na u|^p\ dx + \frac{1}{q} \int_{\sr^N \times \sr^N} \frac{|u(x)- u(y)|^q}{|x-y|^{N+qs}}\ dx\ dy - \frac{1}{1-\delta} \int_{\Om} \left(u^+ \right)^{1-\delta} dx - \int_{\Om} F \left(x, u^+ \right)\ dx,
\end{equation*}
where $\displaystyle u^+ := \max \{u, 0\}$ and $\displaystyle F(x,t):= \int_0^t f(x,\tau)\ d\tau$. A typical example of the nonlinearity $f$ is the power-type function $f(x,t)= t^r$, $t\geq 0$, with a superlinear but at most critical growth rate. More generally, we assume that $f: \overline{\Om} \times \sr \rar \sr$ is a Carath\'eodory function that satisfies the following:
\begin{enumerate}
    \item[\textit{(f1)}] for every $(x,t)\in \overline{\Om} \times \sr^+$, $f(x,t) \geq 0$ and $f(x,0)=0$,
    \item[\textit{(f2)}] there exists $r>0$ satisfying $\max \{p-1, q-1\} <r \leq r^* =\max\{ p^* -1, q_s^* -1 \}$ such that $f(x,t) \leq C_0 (1+t^r)$, for all $(x,t)\in \Om \times \sr^+$, where $C_0 >0$ is a constant and
    \[p^* = \begin{cases}
    \frac{Np}{N-p}&, \text{ if } p<N,\\
    \infty&, \text{ if } p>N,
    \end{cases}\ \ \text{ and }\ \ q_s^* = \begin{cases}
    \frac{Nq}{N-sq}&, \text{ if } sq<N,\\
    \infty&, \text{ if } sq>N.
    \end{cases}\]
\end{enumerate}

\begin{remark}\label{bigger_exponent}
    Note that if $sq<q \leq p$, then $q_s^* < q^* \leq p^*$, so $r^* = p^* -1$. On the other hand, if $sq < p <q$, we cannot generally compare $p^*$ and $q_s^*$, and either exponent could be larger depending on specific values of $p,q$ and $s$.
\end{remark}

The functional $I$ is naturally associated with the perturbed singular problem:
\begin{equation}\label{minimizer_equation}
    \begin{aligned}
        -\p u + \qs u &= \frac{1}{u^{\delta}} + f(x,u) \text{ in } \Om\\
        u>0 \text{ in } \Om,\ u &= 0 \text{ in } \bdry.
    \end{aligned}
\end{equation}
Proposition \ref{minimizer_solution} below provides a clear connection between local minimizers of $I$ and weak solutions of equation \eqref{minimizer_equation}. The main technical difficulty arises due to the singular term $(u^+)^{1-\delta}$, $0<\delta<1$, which prevents the functional $I$ from being of class $C^1$ on the space $\xo (\Om)$. Nevertheless, $I$ is G\^ateaux differentiable (corollary A.3, \cite{giacomoni2007sobolev}) at every $u\in \xo (\Om)$, provided $u$ satisfies the additional condition
\begin{equation}\label{u_greater_than_distance}
    u(x) \geq c d(x) \text{ a.e. in } \Om, \text{ for some } c>0. \nonumber 
\end{equation}

Next, we introduce a suitable auxiliary function, which will play an essential role throughout this section. We consider the purely singular problem:
\begin{equation}\label{minimizer_truncation_equtaion}
    \begin{aligned}
        -\p u + \qs u = \frac{1}{u^{\delta}} \text{ in } \Om\\
        u >0 \text{ in } \Om,\ u = 0 \text{ in } \bdry
    \end{aligned}
\end{equation}
which admits a unique positive solution $\ul\in \coa$, for some $\alpha\in (0,1)$, satisfying $\ul (x) \sim d(x)$ in $\Om$ (see Theorem 3 and remark 1 of \cite{dhanya2026regularity}).

\begin{proposition}\label{minimizer_solution}
    Let $u \in \xo(\Om)$ be a local minimizer of $I$ and satisfy $u(x) \geq c d(x)$ a.e. in $\Om$, for some $c>0$. Then $u$ is a weak solution of \eqref{minimizer_equation} and $u \geq \ul$ a.e. in $\Om$, where $\ul$ is the solution of \eqref{minimizer_truncation_equtaion}.
\end{proposition}

\begin{proof}
    Since $u(x) \geq c d(x)$ a.e. in $\Om$ and is a local minimizer of $I$, it follows that $I$ is G\^ateaux differentiable at $u$ and for every $\varphi\in\test$ and $t>0$ small enough, we have
    \begin{multline*}
        \int_{\Om} |\na u|^{p-2} \na u \cdot \na\varphi\ dx + \int_{\sr^N \times \sr^N} \frac{|u(x) - u(y)|^{q-2} \left(u(x) - u(y) \right) \left(\varphi(x) - \varphi(y) \right)}{|x-y|^{N+sq}}\ dx\ dy\\
        - \int_{\Om} \frac{1}{u^{\delta}} \varphi\ dx - \int_{\Om} f(x,u) \varphi\ dx = \lim_{t\rar 0+} \frac{I(u + t\varphi) - I(u)}{t} \geq 0.
    \end{multline*}
    Replacing $\varphi$ with $-\varphi$ gives the reverse inequality, and hence equality holds in the above expression for every $\varphi\in\test$. Thus, $u$ is a weak solution of \eqref{minimizer_equation}. Also, since $u$ is non-negative, it is a supersolution of equation \eqref{minimizer_truncation_equtaion}. Taking $\left( \ul-u \right)^+$ as the test function in the weak formulation and applying the weak comparison principle, we conclude that $u \geq \ul$ a.e. in $\Om$.
\end{proof}

\begin{remark}\label{minimizer_solution_remark}
    The above proposition also holds if $u$ is a local minimizer of $I$ in the $C^1_0 \omc$-topology and satisfies $u(x) \geq c d(x)$ in $\Om$, for some $c>0$.
\end{remark}

Now, we proceed to the proof of Theorem \ref{Theorem_minimizer}.
\vspace{5pt}

\noindent\textbf{\textit{Proof of Theorem \ref{Theorem_minimizer}:}} Let $u_0$ be a local minimizer of $I$ in $C^1_0 \omc$-topology satisfying $u_0 (x) \geq k_0 d(x)$ in $\Om$. By the remark \ref{minimizer_solution_remark}, $u_0 \geq \ul$ in $\Om$ and $u_0$ solves
\begin{equation}\label{u_0 equation}
    \begin{aligned}
        -\p u_0 + \qs u_0 &= \frac{1}{u_0^{\delta}} + f(x, u_0)  \text{ in } \Om\\
        u_0 &= 0 \text{ in } \bdry
    \end{aligned}
\end{equation}
in the weak sense. Our aim is to show that $u_0$ is a local minimizer of $I$ in $\xo(\Om)$-topology as well. We begin by proving the result in the subcritical case, i.e., when $r< r^*$, and then extend the approach to the critical case of $r=r^*$ by employing truncation techniques.
\vspace{5pt}

\noindent\textbf{\textit{Subcritical case.}} We fix $l\in (r, r^*)$ and for each $\varepsilon>0$, define the set
\[\mathcal{S}_{\varepsilon}:= \left\{ u\in \xo (\Om) : K(u) \leq \varepsilon \right\},\]
where the functional $K$ is given by
\[K(u):= \frac{1}{l+1} \int_{\Om} |u(x) - u_0 (x)|^{l+1}\ dx, \text{ for } u\in \xo (\Om).\]
We consider the following constraint minimization problem:
\[I_{\varepsilon}:= \inf_{u\in \mathcal{S}_{\varepsilon}} I(u).\]
Since $\max \{p-1, q-1\} <r < r^*$ and $0<\delta <1$, it follows from Fatou's lemma that the functional $I$ is weakly lower semicontinuous on $\xo (\Om)$. Moreover, the set $\mathcal{S}_{\varepsilon}$ is a weakly closed subset of $\xo (\Om)$. Thus, the infimum $I_{\varepsilon}$ is attained for some $u_{\varepsilon}\in \mathcal{S}_{\varepsilon}$, that is, $I_{\varepsilon}=I(u_{\varepsilon})$. Next, we make a crucial claim regarding the sequence $\{u_\varepsilon\},$ which we will verify in three steps.
\vspace{5pt}

\noindent\textbf{Claim:} $\{u_\varepsilon\}$ is uniformly bounded in $C^{1}{(\overline{\Omega})}$ and $\|u_\varepsilon -u_0\|_{C^{1}(\overline{\Omega})}\rightarrow 0$ as $\varepsilon\rightarrow 0.$
\vspace{5pt}

\noindent\textit{\textbf{Step 1:}} For every $\varepsilon>0$, $u_{\varepsilon} \geq \ul$ a.e. in $\Om$, where $\ul$ is the unique solution of \eqref{minimizer_truncation_equtaion}.
\vspace{5pt}

Suppose that $V= supp\ (\ul - u_{\varepsilon})^+$ is non-empty. We set
\[w_t = u_{\varepsilon} + t(\ul - u_{\varepsilon})^+ \text{ and } \psi_1 (t) = I (w_t), \text{ for } t\in [0,1]. \]
Then for every $t\in (0,1)$, 
\begin{align*}
    \psi_1' (t) &= \langle I' (w_t),\ (\ul - u_{\varepsilon})^+ \rangle\\
    &\leq \langle -\p w_t + \qs w_t,\ (\ul - u_{\varepsilon})^+ \rangle - \la \int_{\{ \ul > u_{\varepsilon}\}} \frac{1}{\ul^{\delta}} (\ul - u_{\varepsilon})\ dx\\
    &=  \langle -\p w_t + \p \ul,\ (\ul - u_{\varepsilon})^+ \rangle + \langle \qs w_t - \qs \ul,\ (\ul - u_{\varepsilon})^+ \rangle.
\end{align*}
Since $\ul$ solves \eqref{minimizer_truncation_equtaion} and observing that $w_t - \ul = (1-t) (u_{\varepsilon} - \ul) $ on $V$ for $t\in (0,1)$, we obtain the following: 
$$\begin{array}{lll}
    (1-t) \psi_1'(t) &\hspace{-5pt}\leq &\hspace{-5pt} - \displaystyle \int_{\{\ul > u_{\varepsilon}\}} \left( |\na w_t|^{p-2} \na w_t - |\na \ul|^{p-2} \na \ul \right) \cdot \na (w_t - \ul)\\[5mm]
    &\hspace{-5pt} &\hspace{-5pt} - \displaystyle \int_{\sr^N \times \sr^N} \hspace{-20pt}\frac{\mathcal{G}w_t (x,y) - \mathcal{G}\ul (x,y)}{|x-y|^{N+qs}} \left( (w_t - \ul) (x) - (w_t - \ul) (y) \right)\\[5mm]
    & \hspace{-5pt} < & 0 \mbox{ for all } t\in [0,1].
\end{array}
$$
Hence, 
\[\psi_1(1) < \psi_1(0), \text{ or equivalently, } I (w_1) < I (u_{\varepsilon}).\]
Since $\ul \leq u_0$ in $\Om$, it follows that $|w_1 - u_0| \leq |u_{\varepsilon} - u_0|$ a.e. in $\Om$. Thus, $w_1 \in \mathcal{S}_{\varepsilon}$ and this contradicts the fact that $u_{\varepsilon}$ minimizes $I$ over $\mathcal{S}_{\varepsilon}$. Therefore, $V$ must be an empty set, completing the proof of Step 1.
\vspace{5pt}

\noindent\textit{\textbf{Step 2:}} $\{ u_{\varepsilon} \}$ is uniformly bounded in $L^{\infty} (\Om)$.
\vspace{5pt}

We first observe that the minimization of $I$ over the constraint set $\mathcal{S}_{\varepsilon}$ yields, by the Lagrange multiplier rule, the existence of $\mu_\varepsilon \in \mathbb{R}$ such that $I' (u_{\varepsilon})= \mu_{\varepsilon}K'(u_{\varepsilon})$. Equivalently, $u_{\varepsilon}$ satisfies
\begin{equation}\label{case 2 equation}
    \begin{aligned}
        -\p u_{\varepsilon} + \qs u_{\varepsilon} &= \frac{1}{u_{\varepsilon}^{\delta}} + f(x,u_{\varepsilon}) + \mu_{\varepsilon} |u_{\varepsilon} - u_0|^{l-1} (u_{\varepsilon} - u_0) \text{ in } \Om\\
        u_\varepsilon>0 \mbox{ in } \Omega,\ u_{\varepsilon} &= 0 \text{ in } \bdry.
    \end{aligned}
\end{equation}
Consider the function $\psi_2 (t):= I \left(tu_0 + (1-t)u_{\varepsilon} \right),\ t\in [0,1].$
Since the convex combination $tu_0 + (1-t)u_{\varepsilon} \in \mathcal{S}_{\varepsilon}$ for every $t\in [0,1]$ and $u_{\varepsilon}$ minimizes $I$ over $\mathcal{S}_{\varepsilon}$, the one sided derivative $\psi_2'(0) = \langle I' (u_{\varepsilon}), u_0 - u_{\varepsilon} \rangle \geq 0$, which in turn yields $\mu_{\varepsilon} \leq 0$.

Now, we prove that $\{u_\varepsilon\}$ is uniformly bounded in the $L^{\infty}$-norm, considering two cases separately: $\inf_{\varepsilon \in (0,1)} \mu_{\varepsilon} > -\infty$ and $\inf_{\varepsilon \in (0,1)} \mu_{\varepsilon} = -\infty$.
\vspace{5pt}

\noindent\textit{Case I:} Let $\displaystyle\inf_{\varepsilon \in (0,1)} \mu_{\varepsilon} = -\rho > -\infty$.
\vspace{5pt}
    
We first show that $\{ u_{\varepsilon} \}$ is uniformly bounded in $\xo (\Om)$. Since $r<l$, $u_0 \in L^{\infty}(\Om)$ and $K(u_{\varepsilon}) \leq \varepsilon$, we obtain
\begin{equation*}
    \begin{aligned}
        &\int_{\Om} |\na u_{\varepsilon}|^p\ dx + \int_{\sr^N \times \sr^N} \frac{|u_{\varepsilon} (x) - u_{\varepsilon} (y)|^q}{|x-y|^{N+qs}}\ dx\ dy\\
        = &\int_{\Om} u_{\varepsilon}^{1-\delta}\ dx + \int_{\Om} f(x,u_{\varepsilon}) u_{\varepsilon}\ dx + \int_{\Om} \mu_{\varepsilon} |u_{\varepsilon} - u_0|^{l-1} (u_{\varepsilon} - u_0) u_{\varepsilon}\ dx\\
        \leq &\int_{\Om} u_{\varepsilon}^{1- \delta}\ dx + \int_{\Om} C_0 (1+u_{\varepsilon}^r) u_{\varepsilon}\ dx + \rho \int_{\Om} |u_{\varepsilon} - u_0|^l u_{\varepsilon}\ dx\\
        \leq\ & c_1 \int_{\Om} \left(1+ u_{\varepsilon}^{l+1} \right) dx \leq c_2, \text{ independent of } \varepsilon \in (0,1).
    \end{aligned}
\end{equation*}
Thus, the sequence $\{ u_{\varepsilon} \}$ is uniformly bounded in $\xo (\Om)$. Consequently, by lemma \ref{regularity_lemma} and lemma \ref{l_infinity}, $\{ u_{\varepsilon} \}$ is uniformly bounded in $L^{\infty} (\Om)$.
\vspace{5pt}

\noindent\textit{Case II:} Let $\displaystyle\inf_{\varepsilon \in (0,1)} \mu_{\varepsilon} = -\infty$.
\vspace{5pt}

In view of case I, without loss of generality, we may assume $\mu_{\varepsilon} \leq -1$, for every $\varepsilon \in (0,1)$. Then, if we define 
\[g(t_1,x,t_2):= \frac{1}{t_2^{\delta}} + f(x,t_2) + t_1|t_2 -u_0 (x)|^{l-1} (t_2 - u_0 (x)),\]
there exists $c_3>0$ such that $g(t_1, x, t_2)<0$, for all $(t_1,x,t_2) \in (-\infty, -1] \times \Om \times (c_3, +\infty)$. Hence, testing equation \eqref{case 2 equation} with $(u_{\varepsilon} - c_3 )^+$, we get $u_{\varepsilon} \leq c_3$, for every $\varepsilon$. Therefore, combining the above two cases, the proof of Step 2 is now complete.
\vspace{5pt}

\noindent\textit{\textbf{Step 3:}} $\{ u_{\varepsilon}\}$ is uniformly bounded in $\coa$ for some $\alpha \in (0,1)$.
\vspace{5pt}

A key ingredient in proving this uniform H\"older estimate is to establish a uniform upper bound on the term $  -\mu_{\varepsilon} \|u_{\varepsilon} - u_0 \|^l_{\infty}.$ To achieve this, we consider the weak formulation of  \eqref{case 2 equation} and \eqref{u_0 equation}  with the test function $\varphi = |u_{\varepsilon} - u_0|^{\kappa -1} (u_{\varepsilon} - u_0)$, $\kappa \geq 1$, and then subtract the resulting expressions. This yields 
\begin{equation*}
    \begin{aligned}
        0 \leq &\langle -\p u_{\varepsilon} + \qs u_{\varepsilon},\ \varphi \rangle - \langle -\p u_0 + \qs u_0,\ \varphi \rangle\\
        \leq & \langle -\p u_{\varepsilon} + \qs u_{\varepsilon},\ \varphi \rangle - \langle -\p u_0 + \qs u_0,\ \varphi \rangle - \int_{\Om} \left(\frac{1}{u_{\varepsilon}^{\delta}} - \frac{1}{u_0^{\delta}} \right) \varphi \ dx\\
        = & \int_{\Om} \left(f(x, u_{\varepsilon}) - f(x, u_0) \right) \varphi \ dx + \mu_{\varepsilon} \int_{\Om} |u_{\varepsilon} - u_0|^{\kappa +l}\ dx.
    \end{aligned}
\end{equation*}
Since $u_0 \in L^{\infty}(\Omega)$ and from Step 2, the sequence $\{u_\varepsilon\}$ is uniformly bounded in $L^{\infty}(\Om),$ we have 
\begin{equation*}
    \begin{aligned}
        - \mu_{\varepsilon} \int_{\Om} |u_{\varepsilon} - u_0|^{\kappa +l}\ dx &\leq \int_{\Om} \left(f(x, u_{\varepsilon}) - f(x, u_0) \right) \varphi \ dx \leq c_4 |\Om|^{\frac{l}{\kappa+l}} \|u_\varepsilon-u_0\|_{\kappa+l}^\kappa.
    \end{aligned}
\end{equation*}
Thus, $- \mu_{\varepsilon} \|u_{\varepsilon} - u_0 \|^{l}_{\kappa+l} \leq c_4 |\Om|^{\frac{l}{\kappa+l}}$, where $c_4$ does not depend on $\varepsilon$ and $\kappa$. Taking limit as $\kappa \rar +\infty$, the desired uniform upper bound for $-\mu_{\varepsilon} \|u_{\varepsilon} - u_0 \|^l_{\infty}$ is obtained.

Consequently, combining the uniform lower bound for \(u_\varepsilon\) obtained in Step~1 with the above uniform estimate for \(-\mu_\varepsilon\|u_\varepsilon - u_0\|_{\infty}^l\), we deduce that

\[-\p u_{\varepsilon} + \qs u_{\varepsilon} \leq \frac{c_5}{d(x)^{\delta}} \text{ in } \Om.\]
Let $w_0$ be the solution of
\begin{equation}
    \begin{aligned}
        -\p w_0 + \qs w_0 &= \frac{c_5}{d(x)^{\delta}} \text{ in } \Om\\
        w_0 > 0 \text{ in } \Om,\ w_0 &=0 \text{ in } \bdry.
    \end{aligned}
\end{equation}
Then, by weak comparison principle and Theorem 2 of \cite{dhanya2026regularity}, there exists $k_1 >0$ such that $u_{\varepsilon} (x) \leq w_0 (x) \leq k_1 d(x) $ in $\Om$, for every $\varepsilon \in (0,1)$. From Step 1, we already have $u_\varepsilon(x) \geq k_2 d(x)$, for some $k_2>0$. With the uniform upper and lower bounds for $u_\varepsilon$ expressed in terms of the distance function, we are now in a position to invoke the global H\"older regularity theorem (Theorem 2, \cite{dhanya2026regularity}). This ensures that the sequence $\{u_\varepsilon\}$ is uniformly bounded in the $C^{1,\alpha}(\overline{\Omega})$, for some $\alpha\in (0,1)$ and this concludes Step 3. 

Now, the proof of the main claim is rather straightforward. By the Arzela-Ascoli theorem, upto a subsequence $\{ u_{\varepsilon} \}$ converges in $C^1 \omc$ as $\varepsilon \rar 0$. Moreover, the condition $K(u_{\varepsilon}) \leq \varepsilon$ ensures that $u_{\varepsilon} \rar u_0$ in $C^1 \omc$. Finally, to complete the proof in the subcritical case, suppose, on the contrary, that $u_0$ is not a local minimizer of $I$ in the $\mathbb{X}_0 (\Om)$-topology. Then, for every $\varepsilon>0$, there exists $v_{\varepsilon} \in \xo (\Om)$ such that $\| v_{\varepsilon} - u_0 \|_{\mathbb{X}_0 (\Om)} < \varepsilon$ and $I(v_{\varepsilon}) < I(u_0)$. Since $v_{\varepsilon} \in \mathcal{S}_{\varepsilon}$ and $u_\varepsilon$ minimizes $I$ over $S_\varepsilon$, we obtain
\begin{equation}\label{contradiction1}
    I(u_{\varepsilon}) \leq I(v_{\varepsilon}) < I(u_0),\nonumber
\end{equation}
which contradicts that $u_0$ is a local minimizer of $I$ in the $C^1_0 \omc$-topology. Therefore, $u_0$ is also a local minimizer of $I$ in the $\mathbb{X}_0 (\Om)$-topology, and this completes the proof of Theorem \ref{Theorem_minimizer} in the subcritical case.
\vspace{5pt}

\noindent\textbf{\textit{Critical case.}} Define $ \mathcal{S}^*_{\varepsilon}:= \{u\in \xo(\Om) : K_*(u) \leq \varepsilon\}$, where $\varepsilon>0$ and
\[ K_*(u):= \frac{1}{r^* +1} \int_{\Om} |u(x) -u_0 (x)|^{r^* +1}\ dx,\ u\in \xo(\Om).\]
Since $f$ has a critical growth,  here we use a truncation argument. For $j\in \sn$, we consider
\[T_j (t):= \begin{cases}
    -j, \text{ if } t < -j\\
    t, \text{ if } |t|\leq j\\
    j, \text{ if } t > j
\end{cases} \]
and define the truncated energy functional by
\begin{equation*}
    \hspace{-4pt}I_j (u):= \frac{1}{p} \int_{\Om} |\na u|^p\ dx + \frac{1}{q} \int_{\sr^N \times \sr^N} \frac{|u(x) - u(y)|^{q}}{|x-y|^{N+sq}}\ dx\ dy - \frac{1}{1-\delta} \int_{\Om} \left(u^+ \right)^{1-\delta} dx - \int_{\Om} F_j \left(x,u^+ \right) dx,
\end{equation*}
where $f_j(x,t):= f(x, T_j (t))$ and $F_j (x,t):= \int_0^t f_j (x,t)\ dt$.

Now, suppose the conclusion of Theorem \ref{Theorem_minimizer} is not true. Then for every $\varepsilon>0$, there exists $v_{\varepsilon} \in \xo(\Om)$ such that $\|v_{\varepsilon} - u_0\|_{\xo(\Om)} < \varepsilon$ and
\begin{equation}
    I(v_{\varepsilon})<I(u_0).
\end{equation}
Note that $v_{\varepsilon} \in \mathcal{S}^*_{\varepsilon}$ and by the dominated convergence theorem, we have
\[\int_{\Om} F_j (x,v_{\varepsilon})\ dx \rar \int_{\Om} F (x,v_{\varepsilon})\ dx, \text{ as } j\rar\infty.\]
Fix an $\varepsilon>0$ and choose $\vartheta \in (0, I(u_0) - I(v_{\varepsilon}))$. Then, for all $j$ large enough
\begin{equation}\label{delta_epsilon}
    \left| I_j (v_{\varepsilon}) - I(v_{\varepsilon}) \right| = \left|\int_{\Om} F_j (x,v_{\varepsilon})\ dx - \int_{\Om} F(x, v_{\varepsilon})\ dx \right|< \vartheta.
\end{equation}
Now, since $I_j$ is weakly lower semicontinuous on $\xo(\Om)$, let $u_{j,\varepsilon}$ be the minimizer of $I_j$ over $\mathcal{S}^*_{\varepsilon}$. Since $v_{\varepsilon} \in \mathcal{S}^*_{\varepsilon}$, using \eqref{delta_epsilon} and definition of $\vartheta$, we thus have
\begin{equation}\label{critical_main_inequality}
    I_j (u_{j,\varepsilon}) \leq I_j (v_{\varepsilon}) < I(v_{\varepsilon}) + \vartheta < I(u_0) = I_j (u_0), \text{ for all } j\geq j_0.
\end{equation}
Now, proceeding similar to the subcritical case for the functional $I_j,$ we obtain that $u_{j,\varepsilon}$ solves the following equation:
\begin{equation}
    \begin{aligned}
        -\p u_{j,\varepsilon} + \qs u_{j,\varepsilon} &= \frac{1}{u_{j,\varepsilon}^{\delta}} + f_j \left(x,u_{j,\varepsilon} \right) + \mu_{j,\varepsilon} |u_{j,\varepsilon} - u_0|^{r^* - 1} (u_{j,\varepsilon} - u_0) \text{ in } \Om\\
        u_{j,\varepsilon} &= 0 \text{ in } \bdry.
    \end{aligned}
\end{equation}

By lemma \ref{l_infinity}, $\{u_{j,\varepsilon}\}$ is uniformly bounded in $L^{\infty}(\Om)$. We fix a $j$ large enough so that $I \left(u_{j,\varepsilon} \right) = I_j \left(u_{j,\varepsilon} \right).$ Consequently, for a similar reason as in the subcritical case, $\{u_{j,\varepsilon}\}_{\varepsilon>0}$ is uniformly bounded in $C^1 \omc$ and $u_{j,\varepsilon} \rar u_0$ in $C^1 \omc$ as $\varepsilon \rar 0$. Thus, for $\varepsilon>0$ sufficiently small, from \eqref{critical_main_inequality} we have
\[I \left(u_{j,\varepsilon} \right) = I_j \left(u_{j,\varepsilon} \right) < I(u_0),\]
which contradicts that $u_0$ is a local minimizer of $I$ in $C^1_0 \omc$-topology. Hence, our assumption is wrong and $u_0$ minimizes $I$ locally in $\xo(\Om)$-topology as well. This completes the proof of Theorem \ref{Theorem_minimizer}.\hfill\qed

We conclude this section with a corollary which illustrates a combined application of the Hopf-type Strong Comparison Principle and the Sobolev versus H\"older local minimizer result. Furthermore, this corollary will be crucial in establishing the existence and multiplicity of solutions in the subsequent sections.

\begin{corollary}\label{corollary}
    Let $p,q,s$ satisfy the assumptions of Theorem \ref{Theorem_SCP} and $\underline{v}$, $\overline{v} \in \coa$ be respectively sub- and supersolution of
    \begin{equation}\label{corollary_equation}
        \begin{aligned}
            -\p u + \qs u &= \frac{1}{u^{\delta}} + f(x,u) \text{ in } \Om\\
            u > 0 \text{ in } \Om,\ u &=0 \text{ in } \bdry
        \end{aligned}
    \end{equation}
    where $f$ satisfies \textit{(f1)} and \textit{(f2)}. Suppose further that there exists $c>0$ such that $c d(x) \leq \underline{v} \leq \overline{v}$ in $\Om$ and neither $\underline{v}$ nor $\overline{v}$ is a solution of \eqref{corollary_equation}. Then, the problem \eqref{corollary_equation} admits a solution $v_0 \in \coa$ such that
    \begin{itemize}
        \item[\textit{(a)}] $\underline{v} < v_0 < \overline{v}$ in $\Om$, and

        \item[\textit{(b)}] $v_0$ is a local minimizer of the associated energy functional $I$ in the $\xo (\Om)$-topology.
    \end{itemize}
\end{corollary}

\begin{proof}
    We consider the truncated energy functional $\tilde{I} : \xo (\Om) \rar \sr$, given by
    \begin{equation}\label{I_tilde}
        \tilde{I}(u):= \frac{1}{p} \int_{\Om} |\na u|^p\ dx + \frac{1}{q} \int_{\sr^N \times \sr^N} \frac{|u(x) - u(y)|^q}{|x-y|^{N+sq}}\ dx\ dy - \int_{\Om} \tilde{H} \left(x,u^+ \right)\ dx,
    \end{equation}
    where $\tilde{H} (x,t):= \int_0^t \tilde{h}(x,\tau)\ d\tau$ is the primitive of the following cut-off function
    \[\tilde{h} (x,t):= \begin{cases}
        \frac{1}{\underline{v}^{\delta}} + f \left(x,\underline{v} \right),\ \text{ if } t < \underline{v} (x),\vspace{5pt}\\
        \frac{1}{t^{\delta}} + f(x,t),\ \text{ if } \underline{v} (x) \leq t \leq \overline{v} (x),\vspace{5pt}\\
        \frac{1}{\overline{v}^{\delta}} + f \left(x, \overline{v} \right),\ \text{ if } t> \overline{v} (x).
    \end{cases}\]
    Clearly, $\tilde{I}$ admits a global minimizer $v_0 \in \xo (\Om)$, which is a weak solution of
    \begin{equation}\label{h_tilde_equation}
        \begin{aligned}
            -\p v_0 + \qs v_0 = \tilde{h} (x,v_0) \text{ in } \Om\\
            v_0 = 0 \text{ in } \bdry.
        \end{aligned}
    \end{equation}
    From the regularity theory (Theorem 3, \cite{dhanya2026regularity}), it follows that $v_0 \in \coa$, for some $\alpha\in (0,1)$. Moreover, a standard test-function approach yields that $\underline{v} \leq v_0 \leq \overline{v}$ in $\Om$ and hence $v_0$ solves the original problem \eqref{corollary_equation}. Furthermore, since $\underline{v}$ and $\overline{v}$ are not solutions of \eqref{corollary_equation}, from the Strong Comparison Principle (Theorem \ref{Theorem_SCP}), we obtain that $\underline{v} < v_0 < \overline{v}$ in $\Om$.
    \vspace{5pt}

    \noindent\textit{\textbf{Claim:}} $v_0$ is a local minimizer of $I$ in $\xo (\Om)$-topology.
    \vspace{5pt}

    From the Hopf-type Strong Comparison Principle (Theorem \ref{Theorem_SCP}), we have $\frac{\pa \overline{v}}{\pa \nu} <\frac{\pa v}{\pa \nu}<\frac{\pa \underline{v}}{\pa \nu}$ on $\pa  \Omega.$ This ensures the existence of a constant $\epsilon_1 >0$ such that
    \[\underline{v} (x) + \epsilon_1 d(x) \leq v_0 (x) \leq \overline{v} (x) - \epsilon_1 d(x) \text{ in } \Om.\]
    Consider any $v\in C^1_0 \omc$ satisfying $\| v- v_0 \|_{C^1 \omc} < \epsilon_1$. Since $\partial\Omega$ is smooth, there exists a compact set $K$ such that for all $x\in \Omega\setminus K$, we have $d(x)=|x-x_0|$ for some $x_0\in \partial \Omega.$ Hence, as $v=v_0 =0$ on $\pa\Om$, there holds
    $$|v(x)-v_0(x)|\leq \|\nabla v-\nabla v_0\|_{\infty} \|x-x_0\|<\epsilon_1 d(x), \text{ for all } x\in \Om \setminus K.$$
    Let $\displaystyle \epsilon_2= \min_{x\in K}\ \left(\overline{v}(x)-v_0(x) \right)$ and set $0< \epsilon_0 \leq \min\ \{\epsilon_1, \epsilon_2\}$. Then, every $v\in C^1_0 \omc$ with $\|v-v_0\|_{C^{1}\omc}<\epsilon_0$ satisfies $v(x)<\overline{v}(x)$, for all $x\in \Omega$. Hence, 
    \begin{align}\label{I and I tilde}
        \tilde{I}(v) &= \frac{1}{p} \int_{\Om} |\na v|^p\ dx + \frac{1}{q} \int_{\sr^N \times \sr^N} \frac{|v(x) - v(y)|^q}{|x-y|^{N+sq}}\ dx\ dy - \int_{\Om} \tilde{H}(x,v)\ dx\nonumber\\
        &= I(v) + \int_{\Om} \int_0^{v(x)} \left[ \frac{1}{t^{\delta}} + f(x,t) - \tilde{h}(x,t) \right] dt\ dx\nonumber\\
        &= I(v) + \int_{\Om} \int_0^{\underline{v} (x)} \left[ \frac{1}{t^{\delta}} + f(x,t) - \tilde{h}(x,t) \right]\nonumber\\
        &= I(v) + c_2,
    \end{align}
    where $c_2$ is a constant independent of all $v\in C^1_0 \omc$ satisfying $\| v- v_0 \|_{C^1 \omc} < \epsilon_0$. As $v_0$ is a global minimizer of $\tilde{I}$ in $\xo (\Om)$, relation \eqref{I and I tilde} implies that $v_0$ is a local minimizer of $I$ in $C^1_0 \omc$-topology. Consequently, by Theorem \ref{Theorem_minimizer}, $v_0$ is a local minimizer of $I$ in $\xo (\Om)$-topology as well and this completes the proof.
\end{proof}


\section{Existence and non-existence of solutions}\label{section_first_solution}

With all the necessary tools in place, we now investigate the existence of positive solutions to \eqref{P_lambda}. Recall that equation \eqref{P_lambda} is given by
\begin{equation}\tag{$P_{\la}$}
    \begin{aligned}
        -\p u + \qs u &= \frac{\la}{u^{\delta}} + u^r \text{ in } \Om\\
        u>0 \text{ in } \Om,\ u &=0 \text{ in } \bdry
    \end{aligned}
\end{equation}
where $0< \delta <1$, $\max \{ p-1, q-1\} < r \leq r^*= \max \{p^* -1, q_s^* -1 \}$ and $\la$ is a positive parameter. In this section, we prove Theorem \ref{Theorem_first_solution}, identifying a threshold $\Lambda$ that marks the transition between the existence and non-existence of positive solutions to \eqref{P_lambda} as the parameter $\la$ varies.

Consider the energy functional $I_{\la}: \xo(\Om) \rar \sr$ associated with \eqref{P_lambda}, given by
\begin{multline}
    I_{\la}(u):=\frac{1}{p} \int_{\Om} |\na u|^p\ dx + \frac{1}{q} \int_{\sr^N \times \sr^N} \frac{|u(x)- u(y)|^q}{|x-y|^{N+qs}}\ dx\ dy\\
    - \frac{\la}{1-\delta} \int_{\Om} \left(u^+ \right)^{1- \delta} dx - \frac{1}{r+1} \int_{\Om} \left(u^+ \right)^{r+1} dx.
\end{multline}
Also consider the truncated energy functional $\il_{\la}: \xo(\Om) \rar \sr$, defined as
\begin{equation}\label{I lambda lower}
    \il_{\la} (u):=\frac{1}{p} \int_{\Om} |\na u|^p\ dx + \frac{1}{q} \int_{\sr^N \times \sr^N} \frac{|u(x)- u(y)|^q}{|x-y|^{N+qs}}\ dx\ dy - \int_{\Om} H_{\la} \left(x, u^+ \right)\ dx,
\end{equation}
where the function $h_{\la}$ and its primitive $H_{\la}$ are given by
\[ h_{\la} (x,t):= \begin{cases}
    \frac{\la}{\underline{u}_{\la}^{\delta}} + \underline{u}_{\la}^r, \text{ if } t < \underline{u}_{\la}(x)\vspace{5pt}\\
    \frac{\la}{t^{\delta}} + t^r, \text{ if } t \geq \underline{u}_{\la}(x)\\
    \end{cases} \text{ and }\ \ \ \ H_{\la}(x,t):= \int_0^t h_{\la} (x,\tau)\ d\tau,\]
and $\ul_{\la}$ is the unique solution of the purely singular problem
\begin{equation}\label{purely_singular_equation}
    \begin{aligned}
        -\p u + \qs u = \frac{\la}{u^{\delta}} \text{ in } \Om\\
        u > 0 \text{ in } \Om,\ u = 0 \text{ in } \bdry.
    \end{aligned}
\end{equation}
Also note that for $0<\la_1 < \la_2$, weak comparison principle yields that $\ul_{\la_1} \leq \ul_{\la_2}$ in $\Om$. Moreover, $\ul_{\la} \rar 0$ in $\xo(\Om)$ as $\la\rar 0+$.

With these preparations, we now move on to the main results of this section. We begin with a technical lemma, the proof of which is similar to that of lemma 2.2 in \cite{silva2024mixed}.

\begin{lemma}\label{pointwise_convergence_of_gradient}
    Let $\la>0$ and $\{u_n\}$ be a bounded sequence in $\xo(\Om)$ such that $\il_{\la}' (u_n) \rar 0$ as $n\rar\infty$. Then, upto a subsequence, $\na u_n (x) \rar \na u(x)$ a.e. in $\Om$ as $n\rar\infty$, where $u$ is the weak limit of the sequence $\{u_n\}$.
\end{lemma}

\begin{lemma}\label{solution for lambda_0}
    There exists $\la_0>0$ for which $(P_{\la_0})$ admits a positive solution $u_{\la_0}$.
\end{lemma}

\begin{proof}
    We consider the minimization problem
    \[m_{\la} = \inf_{u\in \overline{B_R}}\ \underline{I}_{\la} (u),\]
    where $\overline{B_R} = \left\{u\in \xo(\Om): \|u\|_{\xo (\Om)} \leq R \right\}$. Observe that whenever $\|u\|_{\xo (\Om)}$ is sufficiently small, we have
    \begin{equation}\label{positive_for_small_norm}
        \frac{1}{p} \int_{\Om} |\na u|^p\ dx + \frac{1}{q} \int_{\sr^N \times \sr^N} \frac{\left|u(x)- u(y) \right|^q}{|x-y|^{N+qs}}\ dx\ dy - \frac{1}{r+1} \int_{\Om} |u|^{r+1}\ dx >0.
    \end{equation}
    Thus we choose $R>0$ and $\la = \la_0 >0$ both small enough so that
    \begin{equation}\label{positive_on_bdry}
        \inf_{u\in \pa B_R}\ \underline{I}_{\la_0} (u) >0.
    \end{equation}
    Further, note that $\underline{I}_{\la_0}$ is bounded below on $\overline{B_R}$ and for any $u\in \xo (\Om)$, $u\not\equiv 0$, we have $\underline{I}_{\la_0} (tu) < 0$, for $t$ small. Hence, $-\infty < m_{\la_0}<0$.
    \vspace{5pt}

    \noindent\textbf{\textit{Subcritical case.}} In this case, $\underline{I}_{\la_0}$ is weakly lower semicontinuous on $\xo (\Om)$. Consequently, there exists $u_{\la_0} \in B_R$ such that $\underline{I}_{\la_0} (u_{\la_0}) = m_{\la_0}$ and $u_{\la_0}$ is a weak solution of
    \begin{equation}\label{lambda 0}
        \begin{aligned}
            -\p u_{\la_0} + \qs u_{\la_0} &= h_{\la_0} (x, u_{\la_0}) \text{ in } \Om\\
            u_{\la_0} &= 0 \text{ in } \bdry.
        \end{aligned}
    \end{equation}
    Taking $\left(\underline{u}_{\la_0} - u_{\la_0} \right)^+$ as the test function in the weak formulation of \eqref{purely_singular_equation} and \eqref{lambda 0}, we get $\underline{u}_{\la_0} \leq u_{\la_0}$ in $\Om$, by weak comparison principle. Hence, $u_{\la_0}$ solves the original problem $(P_{\la_0})$.
    \vspace{5pt}

    \noindent\textbf{\textit{Critical case.}} In the critical case, we begin with a minimizing sequence $\{u_n\} \subset B_R$ such that $\il_{\la_0} (u_n) \rar m_{\la_0}$ as $n\rar\infty$. In view of \eqref{positive_on_bdry}, we have $dist\ (u_n, \pa B_R) \geq \vartheta$, for some $\vartheta>0$, which ensures that $u_n \in B_{R'}$ with $0< R' <R$. By Ekeland's variational principle (see corollary 5.12 of \cite{motreanu2014book}), there exist $R_1 >0$ satisfying $R' \leq R_1 <R$ and $\{v_n\}\subset B_{R_1}$ such that
    \begin{equation}\label{Ekeland's_variational_principle}
        \|u_n - v_n\|_{\xo(\Om)} \leq \frac{1}{n},\ \ \il_{\la_0} (v_n) \leq \il_{\la_0} (u_n)\ \ \text{and}\ \ \il_{\la_0}' (v_n) \rar 0\ \ \text{as}\ n\rar\infty.
    \end{equation}
    Thus, $\{v_n\}$ is also a minimizing sequence for $m_{\la_0}$ and
    \begin{equation}\label{equation_from_ekeland}
        -\p v_n + \qs v_n - h_{\la_0} (x, v_n) = o_n (1).
    \end{equation}
    Moreover, since $\{v_n\}$ is bounded, upto a subsequence $v_n \rightharpoonup u_{\la_0}$ in $\xo (\Om)$, for some $u_{\la_0} \in B_{R_1}$. Now, in view of \eqref{equation_from_ekeland}, lemma \ref{pointwise_convergence_of_gradient} gives that, upto a subsequence, $\na v_n (x) \rar \na u_{\la_0} (x)$ a.e. in $\Om$ and applying the Brezis-Lieb lemma \cite{brezis1983lieb}, it follows that
    \begin{equation}
        \begin{aligned}
            \|v_n\|^p_{1,p} &= \|u_{\la_0}\|^p_{1,p} + \|v_n - u_{\la_0}\|^p_{1,p} + o_n(1),\\[3mm]
            [v_n]^q_{s,q} &= [u_{\la_0}]^q_{s,q} + [v_n - u_{\la_0}]^q_{s,q} + o_n(1),\\[3mm]
            \text{ and } \|v_n\|^{r^* +1}_{r^* +1} &=  \|u_{\la_0} \|^{r^* +1}_{r^* +1} + \|v_n - u_{\la_0} \|^{r^* +1}_{r^* +1} +o_n(1).
        \end{aligned}
    \end{equation}
    Therefore, $v_n - u_{\la_0} \in B_R$ and by \eqref{positive_for_small_norm},
    \begin{multline}
        \frac{1}{p} \int_{\Om} |\na v_n - \na u_{\la_0}|^p\ dx + \frac{1}{q} \int_{\sr^N \times \sr^N} \frac{\left|\left(v_n - u_{\la_0}\right)(x)- \left(v_n - u_{\la_0}\right)(y) \right|^q}{|x-y|^{N+qs}}\ dx\ dy\\
        - \frac{1}{r^* +1} \int_{\Om} |v_n - u_{\la_0}|^{r^* +1}\ dx >0.
    \end{multline}
    As a result, we have,
    \begin{equation*}
        \begin{aligned}
            m_{\la_0} &= \il_{\la_0} (v_n) + o_n(1)\\
            &= \il_{\la_0} (u_{\la_0}) + \frac{1}{p} \|v_n - u_{\la_0}\|^p_{1,p} + \frac{1}{q} [v_n - u_{\la_0}]^q_{s,q} - \frac{1}{r^* +1} \|v_n - u_{\la_0}\|^{r^* +1}_{r^* +1} + o_n(1)\\
            &\geq \il_{\la_0} (u_{\la_0}) + o_n(1),
        \end{aligned}
    \end{equation*}
    which implies $\il_{\la_0} (u_{\la_0}) = m_{\la_0}$. Consequently, $u_{\la_0}$ is a local minimizer of $\il_{\la_0}$ and solves equation \eqref{lambda 0} weakly. Finally, arguing similar to the subcritical case, we obtain that $u_{\la_0}$ is a solution of $(P_{\la_0})$ and this completes the proof.
\end{proof}

The above lemma thus ensures the existence of at least one admissible value of the parameter $\la$ for which \eqref{P_lambda} possesses a positive solution. Consequently, the set
\[\mathcal{A}:= \{\la>0 : (P_{\la}) \text{ admits a solution} \}\]
is non-empty. Defining
\begin{equation}\label{definition of threshold}
    \Lambda:= \sup \mathcal{A},
\end{equation}
we immediately have $\Lambda >0$. Next, we show that $\Lambda$ is finite.

\begin{lemma}[Non-existence beyond the threshold]\label{nonexistence}
    The problem \eqref{P_lambda} does not admit any solution for every $\la > \Lambda$.
\end{lemma}

\begin{proof}
    We fix two positive real numbers $r_1$ and $r_2$ such that
    \[0< r_1 < \min\{ p-1, q-1 \},\ \max\{ p-1, q-1 \} < r_2 <r,\]
    and consider the following equation with a convex-concave nonlinearity
    \begin{equation}\label{convex_concave_equation}
        \begin{aligned}
            -\p u + \qs u &= \mu \left(|u|^{r_1 -1}u + |u|^{r_2 -1}u \right) \text{ in } \Om\\
            u &= 0 \text{ in } \bdry.
        \end{aligned}
    \end{equation}
   In \cite{dhanya2025multiplicityresultsmixedlocal}, the authors investigated the existence of solutions to \eqref{convex_concave_equation}. It is established in Theorem 1.2 of \cite{dhanya2025multiplicityresultsmixedlocal} that when $\mu>0$ is sufficiently large, \eqref{convex_concave_equation} admits only the trivial solution. We fix one such $\mu=\mu_0$ large enough so that \eqref{convex_concave_equation} admits only the trivial solution.

    Next, we note that the map $t \mapsto t^{r+\delta} - \mu_0 t^{r_1 + \delta} - \mu_0 t^{r_2 + \delta},\ t>0$, has a global minimum, say $- \tilde{\la}$. Then, for every $\la > \tilde{\la}$, we have
    \begin{equation}\label{inequality_nonexistence}
        \frac{\la}{t^{\delta}} + t^r > \mu_0 \left(t^{r_1} + t^{r_2} \right),\ \forall\ t>0.
    \end{equation}
    Now, we fix some $\la > \tilde{\la}$ and show that \eqref{P_lambda} does not admit any solution. If, on the contrary, \eqref{P_lambda} has a positive solution $u_{\la}$, then, thanks to \eqref{inequality_nonexistence}, $u_{\la}$ becomes a supersolution to equation \eqref{convex_concave_equation} because
    \begin{equation}
        -\p u_{\la} + \qs u_{\la} = \frac{\la}{u_{\la}^{\delta}} + u_{\la}^r > \mu_0 \left(u_{\la}^{r_1} + u_{\la}^{r_2} \right) \text{ in } \Om.
    \end{equation}
    Consequently, similar to the proof of Theorem 1.2 of \cite{dhanya2025multiplicityresultsmixedlocal}, a global minimization of the associated energy functional truncated above by $u_\lambda$ yields a positive solution of \eqref{convex_concave_equation}. This contradicts our choice of $\mu_0$. Therefore, there does not exist any solution to \eqref{P_lambda} for every $\la > \tilde{\la}$. Consequently, the threshold $\Lambda$ defined in \eqref{definition of threshold} is a finite real number. Finally, from the definition of $\Lambda$, it follows that \eqref{P_lambda} does not admit any positive solution for every $\la > \Lambda$ and this completes the proof.
\end{proof}

The following two results together establish that the set $\mathcal{A}$ is an interval and, more precisely, that $\mathcal{A}= (0, \Lambda]$.

\begin{lemma}[Existence of solution below the threshold]\label{multiplicity_local minimizer}
    The problem \eqref{P_lambda} has a positive solution $u_{\la}$ for every $\la\in (0, \Lambda)$. Moreover, $u_{\la}$ is a local minimizer of the energy functional $I_{\la}$ in the $\xo (\Om)$-topology.
\end{lemma}

\begin{proof}
    We fix $\la \in (0,\Lambda)$. Let $\underline{v}=\underline{u}_{\la}$ be the unique solution of \eqref{purely_singular_equation}. Now choose $\la_1 \in (\la, \Lambda)$ for which $(P_{\la_1})$ admits a solution and define $\overline{v} = u_{\la_1}$. Since $\la < \la_1$, we have $\underline{v} \leq \overline{v}$ in $\Om$. Moreover, none of $\underline{v}$ and $\overline{v}$ is a solution of \eqref{P_lambda}. Therefore, by corollary \ref{corollary}, \eqref{P_lambda} admits a positive solution $u_{\la} \in \coa$ such that $\underline{v} \leq u_{\la} \leq \overline{v}$ in $\Om$ and $u_{\la}$ is a local minimizer of the energy functional $I_{\la}$ in the $\xo (\Om)$-topology. This concludes the proof as the argument holds for every $\la \in (0, \Lambda)$.
\end{proof}

\begin{lemma}[Existence of solution at the threshold]\label{solution_limiting case}
    The problem \eqref{P_lambda} admits a solution for $\la = \Lambda$.
\end{lemma}

\begin{proof}
    We consider an increasing sequence $\{\la_n\} \subset \mathcal{A}$ such that $\lim_{n\rar\infty} \la_n = \Lambda$. Let $u_{\la_n}$ be the corresponding solutions of $(P_{\la_n})$ obtained by lemma \ref{multiplicity_local minimizer}. Then for every $n\in\sn$, $u_{\la_n} \geq \underline{u}_{\la_n}$ in $\Om$ and
    \begin{equation}\label{limit_Lambda}
        \langle -\p u_{\la_n} + \qs u_{\la_n},\ \varphi \rangle = \int_{\Om} \frac{\la_n}{u_{\la_n}^{\delta}} \varphi\ dx + \int_{\Om} u_{\la_n}^r \varphi\ dx, \text{ for every } \varphi\in \xo(\Om).
    \end{equation}
    Thus, taking $u_{\la_n}$ itself as the test function, we get
    \begin{equation}\label{first_relation}
        \|u_{\la_n}\|^p_{1,p} + [u_{\la_n}]^q_{s,q} = \la_n \int_{\Om} u_{\la_n}^{1-\delta}\ dx + \int_{\Om} u_{\la_n}^{r+1}\ dx.
    \end{equation}
    Now, from the construction of $u_{\la_n}$, we obtain that
    \[ I_{\la_n}(u_{\la_n}) \leq \il_{\la_n} (\underline{u}_{\la_n} ) <0.\]
    Therefore,
    \begin{equation*}
        \frac{1}{\max \{ p,q\}} \left(\|u_{\la_n}\|^p_{1,p} + [u_{\la_n}]^q_{s,q} \right) - \frac{\la_n}{1-\delta} \int_{\Om} u_{\la_n}^{1-\delta} - \frac{1}{r+1} \int_{\Om} u_{\la_n}^{r+1} \leq I_{\la_n}(u_{\la_n}) < 0
    \end{equation*}
    \begin{equation}
        \text{or, } c_1 \int_{\Om} u_{\la_n}^{r+1}\ dx \leq c_2 \int_{\Om} u_{\la_n}^{1-\delta}\ dx,
    \end{equation}
    where $c_1= \frac{1}{\max \{ p,q\}} - \frac{1}{r+1} >0$ and $c_2 = \Lambda \left(\frac{1}{1-\delta} - \frac{1}{\max \{ p,q\}}\right) >0$. Therefore, from \eqref{first_relation}, we get
    \begin{equation}
        \|u_{\la_n}\|^p_{1,p} + [u_{\la_n}]^q_{s,q} \leq \left(\Lambda + \frac{c_2}{c_1} \right) \int_{\Om} u_{\la_n}^{1-\delta}\ dx \leq c_3 \|u_{\la_n}\|^{1-\delta}_{r^* +1} \leq c_4 \|u_{\la_n}\|^{1-\delta}_{\xo(\Om)}.
    \end{equation}
    Consequently, $\{u_{\la_n}\}$ is bounded in $\xo(\Om)$ and has a weakly convergent subsequence, say $\{u_{\la_n}\}$, by abuse of notation. Let $u_{\la_n} \rightharpoonup u_{\Lambda}$ as $n\rar\infty$. Note that $\langle\il_{\Lambda}' (u_{\la_n}),\ \varphi \rangle \rar 0$ as $n\rar\infty$. Therefore, by lemma \ref{pointwise_convergence_of_gradient}, $\na u_{\la_n} (x) \rar \na u_{\Lambda} (x)$ a.e. in $\Om$. Taking limit in equation \eqref{limit_Lambda} as $n\rar\infty$, we obtain
    \[\langle -\p u_{\Lambda} + \qs u_{\Lambda},\ \varphi \rangle = \int_{\Om} \frac{\Lambda}{u_{\Lambda}^{\delta}} \varphi\ dx + \int_{\Om} u_{\Lambda}^r \varphi\ dx, \text{ for every } \varphi\in \xo(\Om). \]
    Therefore, $u_{\Lambda}$ is a solution of \eqref{P_lambda} for $\la = \Lambda$.
\end{proof}

Finally, we establish the existence of minimal solutions in the range $(0, \Lambda)$.

\begin{lemma}[Existence of minimal solution]\label{minimal solution}
    The problem \eqref{P_lambda} admits a minimal solution for every $0< \la \leq \Lambda$ and the minimal solutions are strictly increasing with respect to $\la$.
\end{lemma}

\begin{proof}
    We use the method of monotone iteration for every fixed $0< \la \leq \Lambda$ to prove this result. Define an iterative sequence $\{w_n\}$ as follows: $w_0 := \underline{u}_{\la}$ and for every $n\in\sn$, $w_n$ is the weak solution of
    \begin{equation}\label{monotone_iteration_scheme}
        \begin{aligned}
            -\p w_n + \qs w_n &- \frac{\la}{w_n^{\delta}} = w_{n-1}^r \text{ in } \Om\\
            w_n > 0 \text{ in } \Om,\ w_n &= 0 \text{ in } \bdry.
        \end{aligned}
    \end{equation}
    Using the fact that the maps $t\mapsto -\la t^{-\delta}$ and $t\mapsto t^r$, $t>0$, are monotone increasing, the weak comparison principle yields that $w_0 \leq w_1$ in $\Om.$ By induction, we obtain that the sequence $\{w_n\}$ is non-decreasing. Moreover, $w_0 \leq u_{\Lambda}$ and $w_n \leq u_{\Lambda}$ for all $n\in\sn$, where $u_{\Lambda}$ is the solution of $(P_{\Lambda})$ obtained in lemma \ref{solution_limiting case}. As a result, it follows from \eqref{monotone_iteration_scheme} that $\{w_n\}$ is bounded in $\xo (\Om)$. Thus, there exists $\hat{u}_{\la} \in \xo(\Om)$ such that $w_n \rightharpoonup \hat{u}_{\la} \in \xo(\Om)$ and $w_n \rar \hat{u}_{\la}$ a.e. in $\Om$. Since $\lim_{n\rar\infty} \langle \il'_{\la} (w_n),\ \varphi \rangle =0$, by lemma \ref{pointwise_convergence_of_gradient}, $\na w_n (x) \rar \na \hat{u}_{\la} (x)$ a.e. in $\Om$. Taking limit in \eqref{monotone_iteration_scheme} as $n\rar\infty$, we obtain that the limiting function $\hat{u}_{\la}$ is a solution of \eqref{P_lambda}. It is easy to see that any solution $u$  of \eqref{P_lambda}, by the weak comparison principle, is bounded below by $w_0$. Thus, $\hat{u}_\lambda\leq u$ in $\Om$ and hence, $u_\lambda$ is the minimal solution of \eqref{P_lambda}.

    Finally, we show that the map $\la \mapsto \hat{u}_{\la}$ is strictly increasing, that is, if $0<\la_1 < \la_2$, then $\hat{u}_{\la_1} < \hat{u}_{\la_2}$ in $\Om$. It is evident from the monotone iteration scheme described above that $\ul_{\la_1} \leq \ul_{\la_2} \leq \hat{u}_{\la_2}$ in $\Om$ and hence $\hat{u}_{\la_1} \leq \hat{u}_{\la_2}$ in $\Om$ as well. Thus, using the non-negativity of $\hat{u}_{\la_2}$ and as $\la_1 < \la_2$, we have
    \[-\p \hat{u}_{\la_1} + \qs \hat{u}_{\la_1} - \frac{\la_1}{\hat{u}_{\la_1}^{\delta}} = \hat{u}_{\la_1}^r \leq \hat{u}_{\la_2}^r \leq -\p \hat{u}_{\la_2} + \qs \hat{u}_{\la_2} - \frac{\la_1}{\hat{u}_{\la_2}^{\delta}} \text{ in } \Om.\]
    Consequently, by the Strong Comparison Principle (Theorem \ref{Theorem_SCP}), $\hat{u}_{\la_1} < \hat{u}_{\la_2}$ in $\Om$ and this completes the proof.
\end{proof}

\noindent\textbf{\textit{Proof of Theorem \ref{Theorem_first_solution}:}} The proof of part $(a)$ follows from lemma \ref{multiplicity_local minimizer} and lemma \ref{solution_limiting case}, while part $(b)$ and part $(c)$ are direct consequences of lemma \ref{minimal solution} and lemma \ref{nonexistence} respectively.\hfill\qed


\section{Existence of a second solution}\label{section_second_solution}

In this final section, we establish the existence of a second positive solution to \eqref{P_lambda} for every $\la \in (0, \Lambda)$. We show that the truncated energy functional $\underline{I}_{\la}$, defined in \eqref{I lambda lower}, admits a critical point $v_{\la}$ satisfying $v_{\la} > \underline{u}_{\la}$ and $u_{\la} \not= v_{\la}$. Since $\underline{I}_{\la}$ is a $C^1$-functional, this ensures that $v_{\la}$ is indeed a solution of \eqref{P_lambda}, distinct from $u_{\la}$. The existence of $v_{\la}$ is obtained via the Mountain Pass Theorem. To this end, we introduce a generalized notion of Palais-Smale sequences for $\il_{\la}$ and establish compactness results for such sequences.

\begin{definition}[Generalized Palais-Smale sequence]
    Let $\mathcal{F}$ be a closed set in $\xo(\Om)$ and $c\in\sr$. We say a sequence $\{v_n\} \subset \xo(\Om)$ is a Palais-Smale sequence for $\ti$ at the level $c$ around $\mathcal{F}$, or a $(PS_{\mathcal{F}, c})$ in short, if
    \[ \lim_{n\rar\infty} dist (v_n, \mathcal{F}) =0,\ \lim_{n\rar\infty} \ti (v_n) = c \text{ and } \lim_{n\rar\infty} \ti' (v_n) =0. \]
\end{definition}

\begin{lemma}\label{PS_sequence_bounded}
    Let $\mathcal{F}$ be a closed set in $\xo(\Om)$ and $c\in\sr$. If $\{v_n\} \subset \xo(\Om)$ is a $(PS_{\mathcal{F}, c})$ sequence for $\ti$, then $\{v_n\}$ is bounded in $\xo(\Om)$ and upto a subsequence $v_n \rightharpoonup v_{\la}$ as $n\rar\infty$, where $v_{\la}$ is a weak solution of \eqref{P_lambda}.
\end{lemma}

\begin{proof}
    Fix any $c\in \sr$ and suppose $\{v_n\}$ is a $(PS_{\mathcal{F}, c})$ sequence for $\ti$. Since $\ul_{\la}$ is positive in $\Om$, we have
    \begin{multline}\label{PS_I_lambda_goes_to_c}
        \frac{1}{p} \|v_n\|_{1,p}^p + \frac{1}{q} [v_n]_{s,q}^q - \int_{\{v_n >\ \underline{u}_{\la}\}} \left( \frac{\la}{1-\delta} v_n^{1-\delta} + \frac{v_n^{r+1}}{r+1} \right) dx\\
        - \int_{\{v_n \leq\ \underline{u}_{\la}\}} \left( \la \ul_{\la}^{-\delta} + \ul_{\la}^r \right) v_n\ dx \leq c +o_n (1)
    \end{multline}
    and
    \begin{equation}\label{PS_I_lambda'_goes_to_0}
        \|v_n\|_{1,p}^p + [v_n]_{s,q}^q - \int_{\{v_n >\ \underline{u}_{\la}\}} \left( \la v_n^{1-\delta} + v_n^{r+1} \right) dx - \int_{\{v_n \leq\ \underline{u}_{\la}\}} \left( \la \ul_{\la}^{-\delta} + \ul_{\la}^r \right) v_n\ dx = o_n (1).
    \end{equation}
    Multiplying \eqref{PS_I_lambda_goes_to_c} by $(r+1)$ and subtracting \eqref{PS_I_lambda'_goes_to_0} from it, we obtain
    \begin{multline}\label{PS_added_condition}
        \left(\frac{r+1-p}{p}\right) \|v_n\|_{1,p}^p + \left(\frac{r+1-q}{q}\right) [v_n]_{s,q}^q - r \int_{\{v_n \leq\ \underline{u}_{\la}\}} \left( \la \ul_{\la}^{-\delta} + \ul_{\la}^r \right) v_n\ dx\\
        + \left( 1 - \frac{r+1}{1-\delta} \right) \int_{\{v_n >\ \underline{u}_{\la}\}} \la v_n^{1-\delta}\ dx \leq (r+1)c + o_n (1).
    \end{multline}
    Note that as $r > \max\{ p-1, q-1 \}$, the coefficients $\frac{r+1-p}{p}$ and $\frac{r+1-q}{q}$ are strictly positive. A further simplification of \eqref{PS_added_condition} yields
    \begin{align*}
        \|v_n\|_{1,p}^p + [v_n]_{s,q}^q &\leq c_1 + c_2 \int_{\{v_n \leq\ \underline{u}_{\la}\}} \left( \la \ul_{\la}^{1-\delta} + \ul_{\la}^{r+1} \right) dx + c_3 \int_{\{v_n >\ \underline{u}_{\la}\}} \la v_n^{1-\delta}\ dx\\[3mm]
        &\leq c_4 + c_5 \|v_n\|_{r^* +1}^{1-\delta} \leq c_4 + c_6 \left( \|v_n\|_{1,p} + [v_n]_{s,q} \right)^{1-\delta}.
    \end{align*}
    Since the constants $c_4$ and $c_6$ are independent of $n$, it follows from the above inequality that the sequence $\{v_n\}$ is bounded in $\xo(\Om)$. Hence there exists a subsequence, still denoting by $\{v_n\}$, that converges weakly in $\xo(\Om)$. Let $v_{\la}$ be the weak limit of $\{v_n\}$ as $n\rar\infty$.

    Now, $\lim_{n\rar\infty} \ti' (v_n) =0$ implies that $v_n$ satisfies
    \begin{equation}\label{PS_equation}
        \begin{aligned}
            -\p v_n + \qs v_n &= h_{\la} (x, v_n) + o_n (1) \text{ in } \Om\\
            v_n &= 0 \text{ in } \bdry.
        \end{aligned}
    \end{equation}
    in the weak sense and by lemma \ref{pointwise_convergence_of_gradient}, $\na v_n (x) \rar \na v_{\la} (x)$ a.e. in $\Om$. Therefore, observing that $\displaystyle h_{\la} \left(x, v_n \right) \leq \ul_{\la}^{-\delta} + \ul_{\la}^r + v_n^r$, where $r\leq r^*$, taking limit in equation \eqref{PS_equation} as $n\rar \infty$, we obtain that $v_{\la}$ is a weak solution of
    \begin{equation}\label{PS_limit_equation}
        \begin{aligned}
            -\p v_{\la} + \qs v_{\la} &= h_{\la} (x, v_{\la}) \text{ in } \Om\\
            v_{\la} &= 0 \text{ in } \bdry.
        \end{aligned}
    \end{equation}
    Taking $(\underline{u}_{\la} -v_{\la})^+$ as a test function, from equations \eqref{purely_singular_equation} and \eqref{PS_limit_equation}, weak comparison principle yields that $v_{\la} \geq \underline{u}_{\la}$ a.e. in $\Om$. Consequently, $v_{\la}$ solves \eqref{P_lambda}.
\end{proof}

\begin{definition}
    Let $\mathcal{F}$ be a closed set in $\xo(\Om)$ and $c\in\sr$. We say that $\ti$ satisfies the Palais-Smale condition at the level $c$ around $\mathcal{F}$ if every $(PS_{\mathcal{F}, c})$ sequence $\{v_n\} \subset \xo(\Om)$ for $\ti$ admits a strongly convergent subsequence in $\xo(\Om)$.
\end{definition}

In the subcritical case, the compactness of the Sobolev embedding ensures that $\ti$ satisfies the Palais-Smale condition at every level $c$. However, in the critical case, the functional satisfies only a local Palais-Smale condition. To proceed, we assume that $u_{\la}$ has the minimum energy among all weak solutions of \eqref{P_lambda}, because if not, we have already found a second solution for \eqref{P_lambda}. Also, note that weak comparison principle confirms that if $w_{\la}$ is any solution of \eqref{P_lambda}, then $w_{\la} \geq \ul_{\la}$ in $\Om$. Therefore, thanks to the definition of the functional $\il_{\la}$, without loss of generality, we assume $\il_{\la} (u_{\la}) \leq \il_{\la} (w_{\la})$ for any weak solution $w_{\la}$ of \eqref{P_lambda}.

\begin{lemma}\label{PS lemma}
    The functional $\ti$ satisfies the Palais-Smale condition in the following settings:
    \begin{itemize}
        \item[\textit{(a)}] If $r< r^*$, at all levels $c \in \sr$,

        \item[\textit{(b)}] If $p\geq q$ and $r=r^* = p^* -1$, at every level $c \in \sr$ satisfying
        \[c< \ti (u_{\la}) + \frac{1}{N} S_0^{N/p},\]
        where $S_0$ is the best constant in the classical Sobolev inequality for $\wop$.
    \end{itemize}
\end{lemma}

\begin{proof}
    Let $c$ be a real number as in the hypotheses and $\{v_n\}$ be a $(PS_{\mathcal{F}, c})$ sequence for $\ti$. Then by lemma \ref{PS_sequence_bounded}, upto a subsequence $v_n \rightharpoonup v_{\la}$ in $\xo(\Om)$, where $v_{\la}$ is a weak solution of \eqref{P_lambda}. We show that, upto a subsequence, $v_n \rar v_{\la}$ strongly in $\xo(\Om)$.
    \vspace{5pt}
    
    \noindent$(a)$ \textit{\textbf{Subcritical case, that is, $r<r^*$.}} We have
    \begin{equation*}
        \langle \ti' (v_n), v_n \rangle = \|v_n\|_{1,p}^p + [v_n]_{s,q}^q - \int_{\Om} h_{\la} \left(x, v_n \right) v_n\ dx.
    \end{equation*}
    Since $r< r^*$ and $\lim_{n\rar\infty} \ti' (v_n) =0$, we have
    \[\lim_{n\rar\infty} (\|v_n\|_{1,p}^p + [v_n]_{s,q}^q) = \int_{\Om} h_{\la} (x,v_{\la}) v_{\la}\ dx = \|v_{\la}\|_{1,p}^p + [v_{\la}]_{s,q}^q,\]
    where the second equality holds as $v_{\la}$ is a weak solution of \eqref{P_lambda}. Due to non-negativity, upto a subsequence, $\lim_{n\rar\infty} \| v_n\|_{1,p}$ and $\lim_{n\rar\infty} [v_n]_{s,q}$ exist. Let $\lim_{n\rar\infty} \| v_n\|_{1,p} = l_p$ and $\lim_{n\rar\infty} [v_n]_{s,q} = l_q$. Then we have
    \begin{equation}\label{norm}
        l_p^p + l_q^q = \|v_{\la}\|_{1,p}^p + [v_{\la}]_{s,q}^q.
    \end{equation}
    We claim that $l_p= \|v_{\la}\|_{1,p}$ and $l_q = [v_{\la}]_{s,q}$. Since $v_n \rightharpoonup v$ in $\xo(\Om)$ and norm is weakly lower semicontinuous, we have
    \[\|v_{\la}\|_{1,p} \leq \liminf \|v_n\|_{1,p} = l_p \text{ and } \|v_{\la}\|_{s,q} \leq \liminf\ [v_n]_{s,q} = l_q.\]
    If $\|v_{\la}\|_{1,p} < l_p,$ or  $[v_{\la}]_{s,q} < l_q$ happens, then it will violate \eqref{norm}. Hence, $\|v_n\|_{\xo(\Om)} \rar \|v_{\la}\|_{\xo(\Om)}$ as $n\rar\infty$. Therefore, upto a subsequence, $v_n \rar v_{\la}$ strongly in $\xo(\Om)$, where $v_{\la}$ is a weak solution of \eqref{P_lambda}.
    \vspace{5pt}
    
    \noindent$(b)$ \textit{\textbf{Critical case, that is, $r=r^*$.}} By the Brezis-Lieb lemma \cite{brezis1983lieb}, we have
    \begin{equation*}
        \begin{aligned}
            &\|v_n - v_{\la}\|^p_{1,p} + [v_n - v_{\la}]^q_{s,q} - \|(v_n - v_{\la})^+\|^{p^*}_{p^*}\\
            = & \|v_n\|^p_{1,p} - \|v_{\la}\|^p_{1,p} + [v_n]^q_{s,q} - [v_{\la}]^q_{s,q} - \|(v_n - v_{\la})^+\|^{p^*}_{p^*} + o_n (1)\\
            = & \langle \il_{\la}' (v_n), v_n - v_{\la} \rangle + \int_{\Om} h_{\la} (x, v_n) (v_n - v_{\la})\ dx - \|(v_n - v_{\la})^+\|^{p^*}_{p^*} + o_n (1)\\
            = & \langle \il_{\la}' (v_n), v_n - v_{\la} \rangle + o_n (1).
        \end{aligned}
    \end{equation*}
    As $\il_{\la}' (v_n) \rar 0$ as $n \rar\infty$, it follows that
    \begin{equation}\label{PS_critical}
        \|(v_n - v_{\la})^+\|^{p^*}_{p^*} = \|v_n - v_{\la}\|^p_{1,p} + [v_n - v_{\la}]^q_{s,q} + o_n (1).
    \end{equation}
    Dividing both sides by $\|v_n - v_{\la}\|^p_{L^{p^*} (\Om)}$, we obtain
    \begin{equation}\label{S_lower_bound}
        \|(v_n - v_{\la})^+\|^{p^* -p}_{p^*} \geq S_0 + \frac{o_n (1)}{\|(v_n - v_{\la})^+\|^p_{p^*}}.
    \end{equation}
    Now, if $v_n \not\rar v_{\la}$ strongly in $\wop$, there exists $k>0$ such that $\|v_n - v_{\la}\|_{1,p} \geq k$, for all $n\in \sn$ and \eqref{S_lower_bound} implies
    \begin{equation}\label{S_power_N/p}
        \|(v_n - v_{\la})^+\|^{p^*}_{p^*} \geq S_0^{N/p} + o_n (1).
    \end{equation}
    Therefore, for $n$ sufficiently large, using \eqref{S_power_N/p} and \eqref{PS_critical}, we have
    \begin{equation}\label{PS_minimum_energy}
        \begin{aligned}
            \frac{1}{N} S_0^{N/p} &\leq \frac{1}{p} \|v_n - v_{\la}\|^p_{1,p} + \frac{1}{q} [v_n - v_{\la}]^q_{s,q} - \frac{1}{p^*} \|(v_n - v_{\la})^+ \|^{p^*}_{p^*} + o_n (1)\\[2mm]
            &= \il_{\la} (v_n) - \il_{\la} (v_{\la}) + o_n (1)\\[2mm]
            &= c - \il_{\la} (v_{\la}) + o_n (1)\\[2mm]
            &< \il_{\la} (u_{\la}) + \frac{1}{N} S_0^{N/p} - \il_{\la} (v_{\la}) .
        \end{aligned}
    \end{equation}
    Since $v_{\la}$ is a weak solution of \eqref{P_lambda}, by our assumption $\il_{\la} (u_{\la}) \leq \il_{\la} (v_{\la})$. Thus, from \eqref{PS_minimum_energy}, we obtain
    \[\frac{1}{N} S_0^{N/p} < \il_{\la} (u_{\la}) + \frac{1}{N} S_0^{N/p} - \il_{\la} (v_{\la}) \leq \frac{1}{N} S_0^{N/p},\]
    which is a contradiction. Therefore, our assumption is wrong and hence $v_n \rar v_{\la}$ strongly in $\wop$.
\end{proof}

Now, we obtain the second solution in the subcritical case using the Mountain Pass Theorem.
\vspace{5pt}

\noindent\textbf{\textit{Proof of Theorem \ref{Theorem_second_solution_subcritical}:}} We first fix some $\la\in (0, \Lambda)$. Since $r> \max\{p-1, q-1 \}$, if $v$ is any function in $\xo(\Om)$ for which $\ti (v) \not= 0$, we have $\lim_{t \rar\infty} \ti (tv) = -\infty$. Again, by lemma \ref{multiplicity_local minimizer}, $u_{\la}$ minimizes $\ti$ locally in $\xo(\Om)$. Therefore, the functional $\ti$ has a mountain pass geometry near $u_{\la}$. We choose $\tilde{v} \in \xo(\Om)\setminus \{0\}$, $\tilde{v} \geq 0$ such that $\ti(\tilde{v}) < \ti (u_{\la})$, define $R_0:= \|\tilde{v} - u_{\la} \|_{\xo(\Om)}$ and fix $R_1 >0$ sufficiently small so that $u_{\la}$ minimizes $\ti$ on the ball $\overline{B_{R_1}(u_{\la})}$. Further, we define the mountain pass level
\[c_0 = \inf_{\gamma \in \Gamma}\ \max_{t\in [0,1]}\ \ti \left(\gamma(t) \right), \]
where $\Gamma:= \left\{ \gamma \in C \left([0,1],\ \xo(\Om) \right): \gamma(0)= u_{\la},\ \gamma(1)= \tilde{v} \right\}$ is the collection of all continuous paths from $u_{\la}$ to $\tilde{v}$. Now, we consider the following two possible cases:
\begin{enumerate}
    \item there exists $R'< R_0$ for which
    \[\inf \left\{ \ti(v): v\in \xo(\Om),\ \| v- u_{\la}\|_{\xo(\Om)} =R' \right\} > \ti (u_{\la}),\]

    \item for every $R<R_0$,
    \[\inf \left\{ \ti(v): v\in \xo(\Om),\ \| v- u_{\la}\|_{\xo(\Om)} =R \right\} \leq \ti (u_{\la}).\]
\end{enumerate}

\noindent\textit{\textbf{Existence of a second positive solution in case (1).}} By lemma \ref{PS lemma}, applying the mountain pass theorem, we obtain a critical point $v_{\la}$ of $\ti$ such that
\[\ti(v_{\la}) = c_0 \geq \inf \left\{ \ti(v): v\in \xo(\Om),\ \| v- u_{\la}\|_{\xo(\Om)} =R' \right\} > \ti (u_{\la}). \]
Therefore, clearly, $u_{\la} \not= v_{\la}$ and $v_{\la}$ is a solution of 
\begin{equation}
    -\p v_{\la} + \qs v_{\la} = h_{\la} (x, v_{\la}) \text{ in } \Om.
\end{equation}
Again, using weak comparison principle (similar to the last part of lemma \ref{PS_sequence_bounded}), we get $v_{\la} \geq \underline{u}_{\la}$ a.e. in $\Om$ so that $v_{\la}$ is, in fact, a solution of \eqref{P_lambda}.
\vspace{5pt}

\noindent\textit{\textbf{Existence of second positive solution in case (2).}} In this zero altitude case, we take $\mathcal{F} = \pa B_{R} (u_{\la})$, $R<R_1$ and following Theorem 1 of Ghoussoub-Preiss \cite{ghoussoub1989mountain}, we obtain a Palais-Smale sequence $\{v_n\}$ for $\ti$ at level $c_0$ around $\mathcal{F}$, that is, $\{v_n\}$ is a sequence in $\xo(\Om)$ satisfying
\[\lim_{n\rar\infty} dist\ (v_n, \mathcal{F}) =0,\ \lim_{n\rar\infty} \underline{I}_{\la} (v_n) = c_0,\ \lim_{n\rar\infty} \underline{I}_{\la} (v_n) =0.\]
Applying lemma \ref{PS lemma}, we obtain $v_{\la}$ such that upto a subsequence $v_n \rar v_{\la}$ as $n\rar\infty$ and $v_{\la}$ is a solution of \eqref{P_lambda}. Further, as $\lim_{n\rar\infty} dist\ (v_n, \pa B_R (u_{\la})) =0$ and $v_n \rar v_{\la}$ strongly in $\xo(\Om)$, we obtain $\|v_{\la} - u_{\la}\|_{\xo(\Om)}=R$. Hence $u_{\la} \not= v_{\la}$ and this completes the proof.\hfill\qed

Next, we move on to obtaining a second solution in the critical case, that is, when $r=r^*$. To this end, we fix $y\in\Om$, $r>0$ small enough and a cut-off function $\phi\in C_c^{\infty} \left(B_{2r}(y) \right)$ satisfying $0\leq \phi \leq 1$ and $\phi \equiv 1$ on $B_r(y)$. Consider the family of Talenti functions
\begin{equation}
    V_{\varepsilon} (x) = C_{N,p} \frac{\varepsilon^{\frac{N-p}{p(p-1)}}}{\left( \varepsilon^{\frac{p}{p-1}} + |x-y|^{\frac{p}{p-1}} \right)^{\frac{N-p}{p}}},\ \varepsilon>0
\end{equation}
where $C_{N,p}$ is a normalization constant and define $U_{\varepsilon} := V_{\varepsilon} \phi$. Then, clearly the functions $\ue$ are supported in $\Om$, or $\ue =0$ in $\bdry$. Now, we derive an energy estimate, which plays a crucial role in verifying that the mountain pass level is consistent with the local Palais-Smale condition.

\begin{lemma}\label{estimate_for_talenti_functions}
    Assume that $2\leq q\leq p$ with $p \in \left( \frac{3N}{N+3},\ 3 \right)$,  
    and
    \[0<s< 1- \frac{1}{q} \left[ \frac{N-p}{p-1} - N \left( 1- \frac{q}{p} \right) \right]. \]
    Then there exist $\varepsilon_0 >0$ and $R_0 \geq 1$ such that for every $\varepsilon \in (0, \varepsilon_0)$, the following hold:
    \begin{itemize}
        \item[\textit{(a)}] $\ti \left(u_{\la} + R\ue \right) = I_{\la} \left( u_{\la} + R\ue\right) < I_{\la} \left( u_{\la} \right)$, for all $R\geq R_0$,
        \item[\textit{(b)}] $\ti \left(u_{\la} + tR_0 \ue \right) = I_{\la} \left( u_{\la} + t R_0 \ue\right) < I_{\la} \left( u_{\la} \right) + \frac{1}{N} S_0^{N/p}$, for all $t\in[0,1]$.
    \end{itemize}
\end{lemma}

\begin{proof}
    For any $R\geq 1$, we have
    \begin{multline}\label{big_expression}
        I_{\la} \left( u_{\la} + tR\ue\right) = \frac{1}{p} \| u_{\la} + tR\ue\|_{1,p}^p + \frac{1}{q} [u_{\la} + tR\ue ]_{s,q}^q\\
        - \frac{\la}{1-\delta} \int_{\Om} (u_{\la} + tR\ue)^{1-\delta}\ dx - \frac{1}{p^*} \|u_{\la} + tR\ue\|_{p^*}^{p^*}.
    \end{multline}
    Using the well-known one-dimensional inequality given in lemma A.4 of \cite{garciaperal1994critical} and H\"older's inequality, we obtain the following estimates:
    \begin{multline}\label{sub_1}
        \|u_{\la} + tR\ue\|_{1,p}^p \leq \|u_{\la}\|_{1,p}^p + (tR)^p \|\ue\|_{1,p}^p + k_1 (tR)^{\zeta_1} \int_{\Om} |\na u_{\la}|^{p-\zeta_1} |\na \ue |^{\zeta_1}\ dx\\
        + p (tR) \int_{\Om} |\na u_{\la}|^{p-2} \na u_{\la} \cdot \na \ue\ dx,
    \end{multline}
    and
    \begin{multline}\label{sub_2}
        [u_{\la} + tR\ue]_{s,q}^q \leq [u_{\la}]_{s,q}^q + (tR)^q [\ue]_{s,q}^q\\
        + k_2 (tR)^{\zeta_2} \int_{\{u_{\la}(x) - u_{\la}(y) \not= 0\}} \frac{|u_{\la} (x) - u_{\la} (y)|^{q-\zeta_2} |\ue (x) - \ue (y)|^{\zeta_2}}{|x-y|^{N+qs}}\ dx\ dy \\
        + q (tR) \int_{\sr^N \times \sr^N} \frac{\mathcal{G} u_{\la} (x,y) \left( \ue (x) - \ue (y) \right)}{|x-y|^{N+qs}}\ dx\ dy,
    \end{multline}
    where $\zeta_1 \in [p-1, 2]$ and $\zeta_2 \in [q-1,2]$. From lemma A.5 of \cite{garciaperal1994critical}, we have
    \[\int_{\Om} |\na \ue|^{m_1}\ dx \leq c_1 \varepsilon^{\frac{N-p}{p(p-1)}m_1}, \text{ if } 1\leq m_1 < \frac{N(p-1)}{N-1}. \]
    Since $p-1 < \frac{N(p-1)}{N-1}$, we fix $\zeta_1$ such that $p-1 < \zeta_1 < \min \left\{ 2,\ \frac{N(p-1)}{N-1} \right\}$. Then, using the regularity of $u_{\la}$, we get that
    \begin{equation}\label{sub_3}
        k_1 (tR)^{\zeta_1} \int_{\Om} |\na u_{\la}|^{p-\zeta_1} |\na \ue |^{\zeta_1}\ dx \leq k_3 (tR)^{\zeta_1} \varepsilon^{\frac{N-p}{p(p-1)}\zeta_1} = o\left( \varepsilon^{\frac{N-p}{p}} \right).
    \end{equation}
    For the analogous nonlocal term, we use lemma 7.5 of \cite{dhanya2025multiplicityresultsmixedlocal}, which gives
    \[ \int_{\sr^N \times \sr^N} \frac{|\ue (x) - \ue (y)|^{m_2}}{|x-y|^{N+qs}}\ dx\ dy \leq c_2 \varepsilon^{\min \left\{N\left(1-\frac{m_2}{p}\right) + m_2 (1-s),\ \frac{m_2(N-p)}{p(p-1)} \right\}},\]
    for $1\leq m_2 \leq p$. Thus, we choose $\zeta_2$ such that $\max \{p-1,\ q-1 \} < \zeta_2 < \min \left\{ 2,\ \frac{N(p-1) +p}{N-p(1-s)} \right\}$, which ensures that
    \begin{equation}\label{sub_4}
        k_2 (tR)^{\zeta_2} \int_{\{u_{\la}(x) - u_{\la}(y) \not= 0\}} \frac{|u_{\la} (x) - u_{\la} (y)|^{q-\zeta_2} |\ue (x) - \ue (y)|^{\zeta_2}}{|x-y|^{N+qs}}\ dx\ dy = o\left( \varepsilon^{\frac{N-p}{p}} \right).
    \end{equation}
    Since $u_{\la}$ is a weak solution of \eqref{P_lambda}, substituting \eqref{sub_1}, \eqref{sub_2}, \eqref{sub_3} and \eqref{sub_4} in \eqref{big_expression}, it follows that
    \begin{multline}\label{big_expression_red_1}
        I_{\la} \left( u_{\la} + tR\ue\right) \leq I_{\la} (u_{\la}) + \left[ \frac{(tR)^p}{p} \|\ue\|^p_{1,p} + \frac{(tR)^q}{q} [\ue]^q_{s,q} \right] + o\left( \varepsilon^{\frac{N-p}{p}} \right)\\
        -\frac{(tR)^{p^*}}{p^*} \|\ue\|^{p^*}_{p^*} -(tR)^{p^* -1} \int_{\Om} \ue^{p^* -1} u_{\la}\ dx + L_1 + L_2,
    \end{multline}
    where
    \begin{align}
        L_1 &= -\frac{1}{p^*} \int_{\Om} \left[ (u_{\la} + tR \ue)^{p^*} - u_{\la}^{p^*} - tR \ue^{p^*} - p^* tR u_{\la} \ue \left( u_{\la}^{p^* -2} + \ue^{p^* -2} \right) \right] dx,\nonumber\\
        \text{and } L_2 &= -\frac{\la}{1-\delta} \int_{\Om} \left[ (u_{\la} + tR \ue)^{1-\delta} - u_{\la}^{1-\delta} - (1-\delta) tR u_{\la}^{-\delta} \ue \right] dx.\nonumber
    \end{align}
    Now, from \cite{garciaperal1994critical} (page 947), we have
    \begin{equation}
        \|\ue\|^p_{1,p} = K_1 + O \left(\varepsilon^{\frac{N-p}{p-1}} \right),\ \|\ue\|^{p^*}_{p^*} = K_2 + O \left(\varepsilon^{\frac{N}{p-1}} \right), \text{ where } K_1 = S_0 K_2^{p/p^*}.
    \end{equation}
    Since $0<s< 1- \frac{1}{q} \left[ \frac{N-p}{p-1} - N \left( 1- \frac{q}{p} \right) \right]$, using lemma 7.5 of \cite{dhanya2025multiplicityresultsmixedlocal}, we get
    \begin{equation}
        [\ue]^q_{s,q} = o\left( \varepsilon^{\frac{N-p}{p}} \right).
    \end{equation}
    As $u_{\la} > 0$ in $\Om$, proceeding similar to \cite{bhakta2025quasilinearproblemsmixedlocalnonlocal} (pages 19 and 22), we obtain
    \begin{equation}
        \int_{\Om} \ue^{p^* -1} u_{\la}\ dx \geq k_4 \varepsilon^{\frac{N-p}{p}}.
    \end{equation}
    Finally, from the estimates proved in \cite{garciaperal1994critical} (pages 946 and 949) and \cite{giacomoni2007sobolev} (pages 137-138), we get
    \begin{equation}
        L_1 + L_2 =  o\left( \varepsilon^{\frac{N-p}{p}} \right).
    \end{equation}
    Consequently, \eqref{big_expression_red_1} simplifies to
    \begin{equation}
        I_{\la} \left( u_{\la} + tR\ue\right) \leq I_{\la} (u_{\la}) + \frac{(tR)^p}{p} K_1 -\frac{(tR)^{p^*}}{p^*} K_2 - (tR)^{p^* -1} k_4 \varepsilon^{\frac{N-p}{p}} + o\left( \varepsilon^{\frac{N-p}{p}} \right).
    \end{equation}
    It is evident from the above inequality that for $t=1$, there exist $\varepsilon_0$ sufficiently small and $R_0$ sufficiently large such that for all $\varepsilon \in (0, \varepsilon_0)$ and $R\geq R_0$, we have
    \[ \frac{R^p}{p} K_1 -\frac{R^{p^*}}{p^*} K_2 - R^{p^* -1} k_4 \varepsilon^{\frac{N-p}{p}} + o\left( \varepsilon^{\frac{N-p}{p}} \right) < 0. \]
    As a result, $I_{\la} \left( u_{\la} + R\ue\right) < I_{\la} (u_{\la})$, for all $\varepsilon \in (0, \varepsilon_0)$ and $R\geq R_0$, proving part $(a)$ of the result. Finally, proceeding similar to \cite{bhakta2025quasilinearproblemsmixedlocalnonlocal} (page 22), we establish part (b), i.e.,
    \[I_{\la} \left( u_{\la} + t R_0 \ue\right) < I_{\la} \left( u_{\la} \right) + \frac{1}{N} S_0^{N/p}, \text{ for all } t\in[0,1]. \]
     This completes the proof.
\end{proof}

\begin{remark}
    When $p=2$, the condition $2 \leq q \leq p$ leads to $q=2$ as well. The corresponding admissible range for $s$ then becomes $0<s< \frac{4-N}{2}$, which is compatible only when $N<4.$ This  yields $0<s<\frac{1}{2}$ when $N=3.$ Notably, this coincides with the range obtained by Biagi et. al. \cite{biagi2024critical} in the linear case, thereby demonstrating the consistency of our approach while at the same time extending the conclusions to the nonlinear case.
\end{remark}

\noindent\textbf{\textit{Proof of Theorem \ref{Theorem_second_solution_critical}:}} We proceed similar to the proof of Theorem \ref{Theorem_second_solution_subcritical}. Here, thanks to part $(a)$ of lemma \ref{estimate_for_talenti_functions}, we choose $\tilde{v} = u_{\la} + R_0 \ue$, where $R_0$ is obtained in lemma \ref{estimate_for_talenti_functions}. We define the mountain pass level
\[c_0 = \inf_{\gamma \in \Gamma}\ \max_{t\in [0,1]}\ \ti \left(\gamma(t) \right), \]
where $\Gamma:= \left\{ \gamma \in C \left([0,1],\ \xo(\Om) \right): \gamma(0)= u_{\la},\ \gamma(1)= \tilde{v}= u_{\la} + R_0 \ue \right\}$ is the collection of all paths from $u_{\la}$ to $\tilde{v} = u_{\la} + R_0 \ue$. Again, we have the following two possible cases:
\begin{enumerate}
    \item there exists $R'< R_0$ for which
    \[\inf \left\{ \ti(v): v\in \xo(\Om),\ \| v- u_{\la}\|_{\xo(\Om)} =R' \right\} > \ti (u_{\la}),\]

    \item for every $R<R_0$,
    \[\inf \left\{ \ti(v): v\in \xo(\Om),\ \| v- u_{\la}\|_{\xo(\Om)} =R \right\} \leq \ti (u_{\la}).\]
\end{enumerate}
In case (1), since $\gamma (t) = u_{\la} + t R_0 \ue$, $t\in[0,1]$, is a path from $u_{\la}$ to $u_{\la} + R_0 \ue$, it follows from part $(b)$ of lemma \ref{estimate_for_talenti_functions} that
\[c_0 \leq \max_{t\in [0,1]}\ \ti \left(\gamma (t) \right) = \max_{t\in [0,1]}\ \ti \left(u_{\la} + t R_0 \ue \right) < \ti \left( u_{\la} \right) + \frac{1}{N} S_0^{N/p}.\]
Therefore, in view of lemma \ref{PS lemma}, we are in a position to apply the mountain pass theorem and obtain a critical point $v_{\la}$ for $\ti$. The condition
\[\ti(v_{\la}) = c_0 \geq \inf \left\{ \ti(v): v\in \xo(\Om),\ \| v- u_{\la}\|_{\xo(\Om)} =R' \right\} > \ti (u_{\la})\]
further ensures that $u_{\la} \not= v_{\la}$.

In case (2), following arguments similar to those in Theorem \ref{Theorem_second_solution_subcritical}, we obtain a second critical point $v_{\la}$ for $\ti$. Now, it only remains to show that $v_{\la}$ solves \eqref{P_lambda}. As $\ti$ is a $C^1$-functional on $\xo(\Om)$ and $v_{\la}$ is a critical point of $\ti$, $v_{\la}$ is a weak solution of
\begin{equation}
    -\p v_{\la} + \qs v_{\la} = h_{\la} (x, v_{\la}) \text{ in } \Om.
\end{equation}
Using weak comparison principle (similar to the last part of lemma \ref{PS_sequence_bounded}), we get $v_{\la} \geq \underline{u}_{\la}$ a.e. in $\Om$ and hence $v_{\la}$ is a solution of \eqref{P_lambda}. This completes the proof.\hfill\qed


\section*{Acknowledgements}
R. Dhanya was supported by SERB MATRICS grant MTR/2022/000780 and Sarbani Pramanik was supported by Prime Minister's Research Fellowship.


\appendix

\section{Infinite semipositone problem}\label{appendix_infinite_semipositone}

In this section, we prove Proposition \ref{infinite_semipositone_existence}, that is, the existence of a positive solution to the following infinite semipositone problem:
\begin{equation}\tag{$Q_{\theta}$}
\begin{aligned}
    -div \left(A_0 (x) \na z \right) &= z^{\si} - \frac{\theta}{z^{\gamma}} \text{ in } U\\
    z > 0 \text{ in } U,\ z &= 0 \text{ on } \pa U
\end{aligned}
\end{equation}
where $U$ is a smooth, bounded domain in $\sr^N$, $\si, \gamma \in (0, 1)$ and $\theta>0$ is small. The matrix $A_0(x) = \left(a^{(0)}_{ij}(x)\right)_{N\times N}$ is symmetric and uniformly positive definite on $U$, with coefficients $a^{(0)}_{ij}$ $(i,j = 1,2,\dots, N)$ being H\"older continuous in $U$. Thus, for every $x\in U$, the maps $(\xi, \eta) \mapsto \langle \xi, \eta \rangle_{A_0}:= A_0(x) \xi \cdot \eta$ and $\xi \mapsto |\xi|_{A_0}:= \sqrt{A_0(x)\xi \cdot \xi}$, where $\xi, \eta \in \sr^N$, respectively define an inner product and the corresponding norm on $\sr^N$.

Before proceeding, we introduce the relevant function spaces. For the distance function $d(x)$ in $U$, we define the space
\[C_{d} (U): = \left\{ u\in C_0 (\overline{U}) : \exists\ c>0 \text{ such that } |u(x)|\leq c d(x), \forall\ x\in U \right\},\]
which is an ordered Banach space endowed with the norm $\|\frac{u}{d}\|_{\infty}.$ The positive cone of $C_{d} (U)$ is
\[C_d^+ (U): = \left\{ u\in C_d (U) : u \geq 0 \text{ in } U \right\},\]
which has a non-empty interior, given by
\[int \ C_d^+ (U): = \left\{ u\in C_d (U) : \inf_{x\in U} \frac{u(x)}{d(x)} >0 \right\}.\]

We establish the existence of a positive solution to \eqref{Q theta} which belongs to the interior of $C_d^+(U).$  We employ a bifurcation argument inspired by  \cite{giacomoni2019existence} to prove the existence of such a solution. At first, we consider the sublinear problem, corresponding to the case of $\theta = 0$, given by
\begin{equation}\tag{$Q_0$}\label{Q zero}
    \begin{aligned}
        -div \left(A_0 (x) \na z \right) &= z^{\si} \text{ in } U\\
        z >0 \text{ in } U,\ z &= 0 \text{ on } \pa U.
    \end{aligned}
\end{equation}
Since the operator is uniformly elliptic in $U$, the standard theory of general elliptic operators is readily applicable. As a result, \eqref{Q zero} admits a positive solution $z_0 \in C^{1, \alpha} (\overline{U})$, for some $\alpha \in (0,1)$ and by Theorem 1.1 of \cite{de2015hopf}, $\frac{\pa z_0}{\pa\nu}<0$ on $\pa U$, i.e., $z_0 \in \text{int} \ C_d^+ (U)$, where $\nu$ denotes the outward normal direction. Using generalized Picone's identity from Proposition 3.1, \cite{pinchover2015criticality}, we can show that the obtained solution $z_0$ of \eqref{Q zero} is unique. 

We also consider singular problems of the form
\begin{equation}\label{aux equn singular}
    \begin{aligned}
        -div \left(A_0(x) \na z \right) &= h(x) \text{ in } U\\
        z &= 0 \text{ on } \pa U
    \end{aligned}
\end{equation}
where $h\in L^{\infty}_{loc}(U)$ and satisfies the following growth condition:
\[\exists\ c>0 \text{ and } 0<\varrho<1 \text{ such that } 0\leq h(x) \leq \frac{c}{d(x)^{\varrho}}, \text{ for a.e. } x\in U.\]
The equation \eqref{aux equn singular} admits a unique solution $z$ in $H^1_0(U)$. Additionally, if there exists $C>0$ such that $0\leq z(x) \leq Cd(x)$ a.e. in $U$, then following Theorem B.1 of \cite{giacomoni2007sobolev}, the solution $z \in C^{1, \alpha} (\overline{U})$, for some $\alpha \in (0,1)$.

Now, we fix $\mu>0$ and choose $\epsilon >0$ sufficiently small such that $B_{\epsilon} (z_0) \subset \text{int}\ C_d^+ (U)$, where $\ball$ is the open ball in $C_d^+ (U)$ with center at $z_0$ and radius $\epsilon$. We define a solution operator corresponding to \eqref{Q theta}. For convenience, we define $\mathcal{L}_0 z := -div \left(A_0 (x) \na z \right)$ and by writing $z= \li (w)$, we mean $z$ solves $\mathcal{L}_0 z = w$ in $U$ with $z=0$ on $\pa U$.

\begin{definition}\label{definition_of_S}
For every $\theta\in (-\mu,\mu)$ and $z\in B_\epsilon(z_0)$  we  define the solution operator $S : \left\{|\theta| < \mu \right\} \times B_{\epsilon} (z_0) \rar C_d (U) $ by $\displaystyle{S(\theta, z):= z - \mathcal{L}_0^{-1} \left(z^{\si} - \frac{\theta}{z^{\gamma}} \right)}$.
\end{definition}
From the above discussion, clearly $\li \left(z^{\si}\right),\ \li \left( \frac{1}{z^{\gamma}} \right) \in C_d (U)$ for every $z\in \ball$. Thus, the map $S$ is well-defined. We first prove the Fr\'echet differentiability of the map $S$. The central idea of the proof follows the approach of \cite{giacomoni2019existence}. However, due to the difference in the operator, we use a slightly modified argument for the convergence of the sequences of the form $\{\li (z_k)\}$, as compared to \cite{giacomoni2019existence}.

\begin{lemma}\label{G cont diff}
    The map $S: \{|\theta| < \mu \} \times B_{\epsilon} (z_0) \rar C_d (U)$ is continuously Fr\'echet differentiable.
\end{lemma}

\begin{proof}
    \textit{\textbf{Continuity of $S$.}} Let $z_k, z \in \ball$, $|\tau_k| < \mu $ such that $\| z_k -z \|_{C_d (U)} + |\tau_k | \rar 0 $ as $k\rar\infty$. Then
    \begin{align}\label{continuity 1}
        &\left|S(\theta + \tau_k, z_k) - S(\theta, z) \right|\nonumber\\
        = &\left|(z_k -z) - \li \left(z_k^{\si} - z^{\si} \right) + \li \left( \frac{\theta+ \tau_k}{z_k^{\gamma}} - \frac{\theta}{z^{\gamma}} \right)\right|\nonumber\\
        \leq &\|z_k -z\|_{C_d (U)} d(x) + \left|\li \left(z_k^{\si} - \frac{\theta}{z_k^{\gamma}} \right) - \li \left(z^{\si} - \frac{\theta}{z^{\gamma}} \right) \right| + \left|\li \left(\frac{\tau_k}{z_k^{\gamma}} \right)\right|.
    \end{align}
    Let $\tilde{z}_k := z_k^{\si} - \frac{\theta}{z_k^{\gamma}}$, $k\in \sn$ and $\tilde{z}:= z^{\si} - \frac{\theta}{z^{\gamma}}$. Then there exists $c_0>0$ such that $\left|\tilde{z}_k (x) \right| \leq \frac{c_0}{d(x)^{\gamma}}$ for all $x\in U$ and for all $k\in\sn$. From the standard regularity theory of singular problems, the sequence $\left\{ \li (\tilde{z}_k)\right\}$ is bounded in $C^{1, \alpha} (\overline{U}) $, for some $\alpha\in (0,1)$. Due to the embedding $C^{1, \alpha}_0 (\overline{U}) \Subset C^1_0 (\overline{U}) \hookrightarrow C_d (U)$, upto a subsequence, $\left\{\li (\tilde{z}_k)\right\}$ converges to $\li (\tilde{z})$ in $C^1_0 (\overline{U})$ as $k\rar\infty$. In fact, thanks to the uniqueness of $\li (\tilde{z})$, every subsequence of $\left\{\li (\tilde{z}_k ) \right\}$ has a further subsequence that converges to $\li (\tilde{z})$ in $C_d (U)$ as $k\rar\infty$. Thus, we obtain
    \begin{equation}\label{continuity 2}
        \left|\li \left(z_k^{\si} - \frac{\theta}{z_k^{\gamma}} \right) - \li \left(z^{\si} - \frac{\theta}{z^{\gamma}} \right) \right| \leq O\left(\|z_k - z\|_{C_d (U)}\right) d(x).
    \end{equation}
    Further, since the sequence $\left\{\li \left(\frac{1}{z_k^{\gamma}} \right)\right\}$ is bounded in $C_d (U)$, we have
    \begin{equation}\label{continuity 3}
        \left|\li \left(\frac{\tau_k}{z_k^{\gamma}} \right)\right| \leq O\left(|\tau_k|\right) d(x).
    \end{equation}
    Combining \eqref{continuity 1}, \eqref{continuity 2} and \eqref{continuity 3}, we have
    \[ \frac{\left|S(\theta + \tau_k, z_k) - S(\theta, z) \right|}{d(x)} \leq O\left(\|z_k - z\|_{C_d (U)} + |\tau_k|\right). \]
    Consequently, the map $S: \{|\theta| < \mu \} \times B_{\epsilon} (z_0) \rar C_d (U)$ is continuous.
    \vspace{5pt}

    \noindent\textit{\textbf{Continuous Fr\'echet differentiability of $S$.}} With a reasoning similar to the continuity part, we can show that for $z\in \ball$ and $\varphi \in C_d (U)$,
    \[\lim_{t\rar 0+} \frac{|S(\theta, z + t \varphi) - S(\theta, z)|}{t} = \varphi - \si \li \left(z^{\si -1} \varphi \right) - \theta \li \left(\gamma z^{-\gamma -1} \varphi \right). \]
    Thus, the map $S(\theta, \cdot)$ is G\^ateaux differentiable at every $z\in \ball$ and the G\^ateaux derivative is given by
    \[\pa_z S(\theta, z)(\varphi):= \varphi - \si \li \left(z^{\si -1} \varphi \right) - \theta \li \left(\gamma z^{-\gamma -1} \varphi \right),\ \varphi \in C_d (U). \]
    Now, we have
    \begin{multline}\label{Frechet 1}
        |S(\theta, z+\varphi) - S(\theta, z) - \pa_z S(\theta, z)(\varphi)|\nonumber\\
        \leq \left|\li \left( (z+\varphi)^{\si} - z^{\si} - \si z^{\si -1} \varphi \right)\right| + \left| \theta \li \left( \frac{1}{(z+\varphi)^{\gamma}} - \frac{1}{z^{\gamma}} + \frac{\gamma \varphi}{z^{\gamma +1}} \right) \right|.
    \end{multline}
     Again, using an argument similar to the previous part, we obtain
     \[ \left\|S(\theta, z+\varphi) - S(\theta, z) - \pa_z S(\theta, z)(\varphi)\right\|_{C_d (U)} \rar 0 \text{ as } \|\varphi\|_{C_d (U)} \rar 0. \]
     Also, it is straightforward to see that $\pa_{\theta} S(\theta, z) $ exists and is continuous, where $\pa_{\theta} S(\theta, z):= \li \left( \frac{1}{z^{\gamma}} \right) $. Thus, $S$ is Fr\'echet differentiable.
     
     Finally, we establish the continuity of $\pa_z S(\theta, z)$. Let $\{z_k\}$ be a sequence in $\ball$ such that $\| z_k - z \|_{C_d (U)} \rar 0$ as $k\rar\infty$. Then
     \begin{equation}
         \left| \left(\pa_{z} S(\theta, z_k) - \pa_z S(\theta, z) \right) (\varphi) \right| \leq \left| \si \li \left( \left(z_k^{\si -1} - z^{\si -1}\right) \varphi \right) \right| + \left| \gamma \theta \li \left( \frac{\varphi}{z_k^{\gamma +1}} - \frac{\varphi}{z^{\gamma+1}} \right) \right|.
     \end{equation}
     Hence, again a similar convergence argument yields $\pa_z S(\theta, z)$ is continuous and this completes the proof.
\end{proof}

Thus, for any fixed $(\theta,z)\in \{|\theta|<\mu\} \times \ball$, the map $\pa_z S(\theta,z): C_d (U) \rar C_d (U)$ is given by
\begin{equation}
    \pa_z S(\theta,z)(\varphi):= \varphi- \li \left(\si z^{\si -1}\varphi + \frac{\theta \gamma \varphi}{z^{\gamma +1}} \right), \text{ for } \varphi\in C_d (U).
\end{equation}
We also have $S(0, z_0) =0$ as $z_0$ solves \eqref{Q zero}. Now, our objective is to employ the implicit function theorem to establish the existence of a positive solution to \eqref{Q theta}. To this end, we study an associated eigenvalue problem. Let us define
\[\Lambda_1: = \inf_{\substack{\varphi\in H^1_0 (U) \\ \varphi \not\equiv 0}} \left\{\frac{\int_{U} |\varphi|^2_{A_0}\ dx - \si \int_{U} z_0^{\si -1} \varphi^2\ dx}{\int_{U} \varphi^2\ dx} \right\}.\]
Clearly, the above infimum is attained for some $\psi \in H^1_0 (U)$ with $\psi \geq 0$ in $U$, $\|\psi\|_{L^2 (U)} =1$ and $\psi$ solves the following eigenvalue problem:
\begin{equation}\label{eigenvalue problem}
    \begin{aligned}
        -div \left(A_0 (x) \na \psi \right) &= \Lambda_1 \psi + \si z_0^{\si -1} \psi \text{ in } U\\
        \psi &= 0 \text{ on } \pa U.
    \end{aligned}
\end{equation}
The standard elliptic regularity theory ensures that $\psi \in C^{1,\alpha} (\overline{U})$, for some $\alpha \in (0,1)$. Moreover, by the Strong Maximum Principle of Pucci-Serrin (Theorem 2.5.1, \cite{pucci2007maximum}), we obtain $\psi >0$ in $U$. In the following, we list further key properties of the eigenvalue $\Lambda_1$, relevant to our subsequent analysis.

\begin{lemma}
    For the eigenvalue problem \eqref{eigenvalue problem}, the following hold:
    \begin{itemize}
        \item[(a)] $\Lambda_1$ is the unique principal eigenvalue,
        \item[(b)] any non-negative eigenfunction corresponding to $\Lambda_1$ belongs to $\intcd$,
        \item[(c)] $\Lambda_1 >0$.
    \end{itemize}
\end{lemma}

The proof of this lemma is standard and closely follows the arguments presented in \cite{giacomoni2019existence} for the fractional Laplacian operator. Now, we are in a position to apply the implicit function theorem, which yields a positive solution to \eqref{Q theta}, as stated in the Proposition \ref{infinite_semipositone_existence}.
\vspace{5pt}

\noindent\textbf{\textit{Proof of Proposition \ref{infinite_semipositone_existence}:}} We first show that the map $\pa_z S(0, z_0): C_d (U) \rar C_d (U)$ is invertible. Owing to its linearity, proving $\pa_z S(0, z_0) (\varphi) =0$ implies $\varphi =0$ in $C_d (U)$ is sufficient to conclude that $\pa_z S(0, z_0)$ is injective. Assume $\pa_z S(0, z_0) (\varphi) =0$, that is, $\varphi$ satisfies
\begin{equation*}
    \begin{aligned}
        -div \left(A_0 (x) \na \varphi \right) &= \si z_0^{\si -1} \varphi \text{ in } U\\
        \varphi &= 0 \text{ on } \pa U.
    \end{aligned}
\end{equation*}
If $\varphi \not\equiv 0$ in $C_d (U)$, then testing the above equation with $\varphi$ and using the positivity of $\Lambda_1$, we obtain
\[0 = \int_{U} A_0 (x) \na \varphi \cdot \na \varphi\ dx - \si \int_{U} z_0^{\si -1} \varphi^2\ dx \geq \Lambda_1 \int_{U} \varphi^2\ dx >0,\]
which is a contradiction. Hence, we must have $\varphi =0$, proving that $\pa_z S(0, z_0): C_d (U) \rar C_d (U)$ is injective. Next, since $z_0 \in \intcd$ and thanks to the global $C^{1,\alpha}$ regularity (Theorem B.1 of \cite{giacomoni2007sobolev}), we note that the linear map $\varphi \mapsto \li \left(\si z_0^{\si -1} \varphi \right)$, $\varphi \in C_d (U)$ is compact. Thus, $\pa_z S(0, z_0)$ is a compact perturbation of the identity operator and, being injective, it is therefore invertible.

By implicit function theorem, there exist $0<\mu' < \mu$, $0<\epsilon' < \epsilon$ and a $C^1$-map $g_1 : \{ |\theta| < \mu'\} \rar B_{\epsilon'} (z_0) $ such that $S(\theta, z) = 0$ in $\{ |\theta| < \mu'\} \times B_{\epsilon'} (z_0)$ coincides with the graph of $g_1$. Now, by definition, $S(\theta, z) = 0$ if and only if $z$ is a solution of \eqref{Q theta}. Therefore, for sufficiently small values of $\theta$, \eqref{Q theta} admits a positive solution $z_{\theta} \in int\ C_d^+ (U)$.\hfill\qed


\section{$L^{\infty}$-regularity}\label{appendix_L_infinity_regularity}

In this section, we discuss the $L^{\infty}$-regularity of weak solutions to the following equation
\begin{equation}\label{L_infinity_equation}
    \begin{aligned}
        -\p u + \qs u &= \frac{1}{u^{\delta}} + f(x,u) \text{ in } \Om\\
        u>0 \text{ in } \Om,\ u &= 0 \text{ in } \bdry
    \end{aligned}
\end{equation}
where $0<\delta<1$ and $f$ satisfies \textit{(f1)}, \textit{(f2)} outlined in Section \ref{section_minimizer}. Due to the singular nature of the nonlinearity near the boundary of the domain, we cannot directly follow the standard approach. To overcome this issue, we first observe a useful technical result, which allows us to reduce the regularity analysis to that of a non-singular problem.

\begin{lemma}\label{regularity_lemma}
   If $u\in\xo (\Om)$ is a weak solution to equation \eqref{L_infinity_equation}, then $u$ satisfies
    \[-\p \left((u-1)^+ \right) + \qs \left((u-1)^+ \right) \leq 1 + f\left(x,(u-1)^+ \right) \text{ in } \Om,\]
    in the weak sense.
\end{lemma}

The proof of this lemma is based on an adaptation of the approach from Proposition 5.1 of \cite{giacomoni2019verysingular}, suitably modified to account for the present mixed local-nonlocal setting. The computations related to the nonlocal term closely follow those in Theorem A.1 of \cite{garain2022mixed}. For brevity, we omit the proof here.

Thanks to lemma \ref{regularity_lemma}, it is now sufficient to study the regularity of weak solutions to the following non-singular problem:
\begin{equation}\label{l_infinity_nonsingular}
    \begin{aligned}
        -\p u + \qs u &= f(x,u) \text{ in } \Om\\
        u &=0 \text{ in } \bdry.
    \end{aligned}
\end{equation}
The next lemma establishes a uniform a priori $L^{\infty}$-estimate for weak solutions to the above equation \eqref{l_infinity_nonsingular}, which plays a crucial role in the proof of Theorem \ref{Theorem_minimizer} in Section \ref{section_minimizer}.

\begin{lemma}\label{l_infinity}
    Let $p,q >1$ and $u\in \xo (\Om)$ be a positive weak solution of \eqref{l_infinity_nonsingular}. Then $u\in L^{\infty}(\Om)$ and there exists a constant $M_0 = M_0(p, q,s,\Om,C_0)>0$, independent of $u$, such that
    \[ \|u\|_{\infty} \leq M_0 \left(1+\int_{\Om}|u|^{\beta_1 (r^* +1)}\ dx \right)^{\frac{1}{(r^* +1) (\beta_1 -1)}}, \]
    where $\displaystyle \beta_1 = \begin{cases}
        \frac{p^* -1}{p}+1, \text{ if } r^* = p^* -1,\\[2mm]
        \frac{q_s^* -1}{q}+1, \text{ if } r^* = q_s^* -1.
    \end{cases}$
\end{lemma}

\begin{proof}
    Firstly, we observe that when $p\geq N$, the standard embedding theorems ensure that $u\in L^t(\Om)$ for every $1\leq t < \infty$. Consequently, the right-hand side of equation \eqref{l_infinity_nonsingular}, namely $f(x,u)$, belongs to $L^{t_1}(\Om)$ for some $t_1 > N$. The desired result then follows directly from the global regularity theorem established in \cite{antonini2023global}. Therefore, in what follows, we restrict our attention to the non-trivial case $p<N$.
    
    Furthermore, we split the proof into two cases according to the value of $r^*$, namely $r^* = p^* -1$ and $r^* = q_s^* -1$. We first consider the case $r^* = p^* -1$ and establish the result through a two-step analysis.
    \vspace{5pt}

    \noindent\textit{\textbf{Claim:}} If $r^*=p^*-1,$ then $u\in L^{\beta_1 p^*}(\Om)$, where $\beta_1 = \frac{p^* -1}{p}+1$.
    \vspace{5pt}
    
    For $\beta>1$ and $T>1$, define a function $\phi: \sr\rar\sr$ by
    \[\phi(t) := \begin{cases}
        0, \text{ if } t < 0,\\[1mm]
        t^{(\ba -1)p+1}, \text{ if } 0 \leq t \leq T,\\[1mm]
        \left( (\ba-1)p+1 \right) T^{(\ba -1)p} \left(t-T \right) +T^{(\ba-1)p+1}, \text{ if } t > T.
    \end{cases}\]
    Since $\phi$ is Lipschitz continuous with $\phi(0)=0$, it follows that if $u\in\xo (\Om)$, then $\phi(u)\in\xo (\Om)$ as well. Taking $\phi(u)$ as a test function in \eqref{l_infinity_nonsingular}, we obtain
    \begin{equation}\label{I1 and I2 equation}
        J_1 + J_2 = \int_{\Om} f(x,u) \phi(u)\ dx,
    \end{equation}
    where $\displaystyle J_1 = \int_{\Om} |\na u|^{p-2} \na u \cdot \na \left(\phi(u)\right) dx$ and \[\hspace{-45pt}J_2 = \int_{\sr^N \times \sr^N} \frac{|u(x)-u(y)|^{q-2} \left( u(x)-u(y) \right) \left((\phi \circ u)(x) - (\phi \circ u)(y) \right)}{|x-y|^{N+qs}}\ dx\ dy.\]

    Let us define $\displaystyle \phi_1 (t):= \int_0^t \phi' (\tau)^{1/p}\ d\tau$ and $\displaystyle \phi_2 (t):= \int_0^t \phi' (\tau)^{1/q}\ d\tau,\ t\in\sr$. Both $\phi_1$ and $\phi_2$ are then Lipschitz continuous functions with $\phi_1(0)=\phi_2(0)=0$ and hence $\phi_1(u),\ \phi_2 (u) \in \xo(\Om)$ whenever $u\in\xo(\Om)$. Moreover, 
    \begin{equation}\label{I1 estimate}
        J_1 = \int_{\Om} |\na u|^p \phi'(u)\ dx = \int_{\Om} |\na u|^p \left(\phi_1'(u)\right)^p\ dx=\|\phi_1 (u)\|^p_{1,p}
    \end{equation}
    and by lemma A.2 of \cite{brasco2016second}, we have
    \begin{equation}\label{I2 estimate}
        J_2 \geq \int_{\sr^N \times \sr^N} \frac{\left|\phi_2\left(u(x)\right) - \phi_2\left(u(y)\right)\right|^q}{|x-y|^{N+qs}}\ dx\ dy \geq 0.
    \end{equation}
    Combining \eqref{I1 and I2 equation}, \eqref{I1 estimate}, \eqref{I2 estimate} and applying Sobolev inequality, we obtain
    \begin{equation*}
        \|\phi_1 (u)\|^p_{p^*} \leq c_1 \|\phi_1 (u)\|^p_{1,p}\leq c_1 \int_{\Om} f(x,u)\phi(u)\ dx.
    \end{equation*}
    Next, we observe that $\phi_1 (t) = \frac{1}{\beta} \left((\beta -1)p +1 \right)^{1/p} \psi (t)$, where $\psi : \sr \rar \sr$ is the function defined by
    \[\psi(t) := \begin{cases}
        0, \text{ if } t < 0,\\[1mm]
        t^{\ba}, \text{ if } 0 \leq t \leq T,\\[1mm]
        \ba T^{\ba -1} (t-T) +T^{\ba}, \text{ if } t > T.
    \end{cases}\]
    Consequently, using the growth condition \textit{(f2)}, we have
    \[\|\psi (u)\|^p_{p^*} \leq c_2 \beta^p \int_{\Om} \left(1+ u^{p^* -1}\right)\phi(u)\ dx,\]
    where the constant $c_2 = c_2 \left(p, q,s,\Om,C_0 \right) >0$ is independent of $u$. Furthermore, noting that $\phi(t) \leq \left((\beta -1)p +1 \right) t^{(\beta -1)p +1}$ and $t^{p-1} \phi(t) \leq \left((\beta -1)p +1 \right) \left(\psi (t)\right)^p$, we deduce the following key estimate:
    \begin{equation}\label{iteration_starting_point}
        \|\psi (u)\|^p_{p^*} \leq c_2 \beta^p \left((\beta -1)p +1 \right) \left( \int_{\Om} u^{(\beta -1)p +1}\ dx + \int_{\Om} u^{p^* -p}\left(\psi (u)\right)^p dx \right).
    \end{equation}
    Now, choosing $\ba =\ba_1 := \frac{p^* -1}{p}+1$ gives
    \begin{equation}\label{J1 and J2}
        \|\psi (u)\|^p_{p^*} \leq c_2 \beta_1^p p^* \left( \int_{\Om} u^{p^*} dx + \int_{\Om} u^{p^* -p}\left(\psi (u)\right)^p dx \right).
    \end{equation}
    To estimate the second integral on the RHS, we introduce $R>0$, which is to be chosen later. By H\"older's inequality and observing that $\psi(t) \leq \beta t^{\beta}$, we obtain
    \begin{align}\label{J2}
        \int_{\Om} u^{p^* -p}\left(\psi (u)\right)^p dx\nonumber \leq &\int_{\left\{u\leq R\right\}} \frac{\left(\psi(u)\right)^p}{u^{p-1}} R^{p^* -1}\ dx + \left( \int_{\Om} \left( \psi (u)\right)^{p^*} dx \right)^{p/{p^*}} \left( \int_{\left\{u > R\right\}} u^{p^*} dx \right)^{\frac{p^* -p}{p^*}}\nonumber\\
        \leq &R^{p^* -1} \beta_1^p \int_{\Om} u^{p^*} dx + \left( \int_{\Om} \left( \psi (u)\right)^{p^*} dx \right)^{p/{p^*}} \left( \int_{\left\{u > R\right\}} u^{p^*} dx \right)^{\frac{p^* -p}{p^*}}.
    \end{align}
    Now, thanks to the dominated convergence theorem, we choose $R>0$ sufficiently large so that
    \[\left( \int_{\left\{u > R\right\}} u^{p^*} dx \right)^{\frac{p^* -p}{p^*}} \leq \frac{1}{2c_2 \beta_1^p p^*}.\]
    Substituting the estimate \eqref{J2} in \eqref{J1 and J2}, and absorbing the last term of \eqref{J2} into the LHS of \eqref{J1 and J2}, we arrive at
    \begin{equation*}
        \left(\int_{\Om} \left( \psi(u)\right)^{p^*} dx \right)^{p/p^*} \leq 2c_2 \beta_1^p p^* \left(1 + R^{p^* -1} \beta_1^p \right) \int_{\Om} u^{p^*} dx,
    \end{equation*}
    where the RHS is finite and independent of $T$. Therefore, letting $T\rar\infty$, we conclude that
    \begin{equation}
        \left( \int_{\Om} u^{\beta_1 p^*} dx \right)^{p/p^*} < \infty,
    \end{equation}
    that is, $u\in L^{\beta_1 p^*} (\Om)$ and this establishes our claim.

    Next, in the second step, we follow the Moser iteration approach to improve the integrability of $u$ further. Recalling the estimate in \eqref{iteration_starting_point} and using the properties of the function $\psi$, we have
    \begin{equation}
        \|\psi (u)\|^p_{p^*} \leq c_2 \beta^{2p} \left((\beta -1)p +1 \right) \left( \int_{\Om} u^{(\beta -1)p +1}\ dx + \int_{\Om} u^{(\beta -1)p + p^*} dx \right),
    \end{equation}
    which holds for any $\beta>1$ for which the RHS is finite. Moreover, the RHS being independent of $T$, we may pass to the limit as $T\rar\infty$ and obtain
    \begin{equation}\label{iteration_second_point}
        \left( \int_{\Om} u^{\beta p^*} dx \right)^{p/p^*} \leq c_2 \beta^{2p} \left((\beta -1)p +1 \right) \left( \int_{\Om} u^{(\beta -1)p +1}\ dx + \int_{\Om} u^{(\beta -1)p + p^*} dx \right).
    \end{equation}
    Now, to handle the first integral on the RHS, we apply Young's inequality with conjugate exponents $\frac{b}{a}$ and $\frac{b}{b-a}$, where $a= (\beta -1)p +1$ and $b= (\beta -1)p + p^*$. This yields
    \begin{equation}
        \int_{\Om} u^{(\beta -1)p +1}\ dx \leq \frac{a}{b} \int_{\Om} u^{(\beta -1)p + p^*} dx + \frac{b-a}{b}\ |\Om|.
    \end{equation}
    Substituting this bound into \eqref{iteration_second_point}, we thus obtain
    \begin{equation}
        \left( \int_{\Om} u^{\beta p^*} dx \right)^{p/p^*} \leq c_3 \beta^{2p} \left((\beta -1)p +1 \right) \left( 1 + \int_{\Om} u^{(\beta -1)p + p^*} dx \right),
    \end{equation}
    where $c_3 = c_3 \left(p, q,s,\Om,C_0 \right) >0$ is also independent of $u$. Applying the elementary inequality $(a+b)^p \leq 2^{p-1} (a^p + b^p)$, for any $a,b >0$, we further deduce
    \begin{equation}\label{iteration}
        \left( 1+ \int_{\Om} u^{\beta p^*} dx \right)^{\frac{1}{p^* (\beta -1)}} \leq c_4^{\frac{1}{\beta -1}} \beta^{\frac{2p+1}{p(\beta -1)}} \left( 1+ \int_{\Om} u^{(\beta -1)p + p^*} dx \right)^{\frac{1}{p(\beta -1)}},
    \end{equation}
    for some constant $c_4 = c_4 \left(p, q,s,\Om,C_0 \right) >0$, independent of $u$.
    
    Now, we iterate this relation over a sequence $\{\beta_m\}_{m\geq 1}$ defined recursively by
    \[\ba_1 = \frac{p^* -1}{p}+1,\ (\beta_{m+1} -1)p = (\beta_m -1) p^*,\ m\geq 1.\]
    This yields
    \begin{equation}
        \left( 1+ \int_{\Om} u^{\beta_{m+1} p^*} dx \right)^{\frac{1}{p^* (\beta_{m+1} -1)}} \leq c_4^{\sum_{i=2}^{m+1} \frac{1}{\beta_i -1}} \left( \prod_{i=2}^{m+1} \beta_i^{\frac{1}{\beta_i -1}} \right)^{\frac{2p+1}{p}} \left( 1+ \int_{\Om} u^{\beta_1 p^*} dx \right)^{\frac{1}{p^* (\beta_1 -1)}}.
    \end{equation}
    Noting that $\beta_{i+1} -1 = \left(\frac{p^*}{p}\right)^i (\beta_1 -1),\ i\in \sn$, and $\frac{p}{p^*}<1$, it follows that both $\sum_{i=2}^{\infty} \frac{1}{\beta_i -1}$ and $\prod_{i=2}^{\infty} \beta_i^{\frac{1}{\beta_i -1}}$ are convergent. Thus, there exists a positive constant $M_0 = M_0 \left(p, q,s,\Om,C_0 \right)$ such that for every $m\in\sn$,
    \[\left( 1+ \int_{\Om} u^{\beta_{m+1} p^*} dx \right)^{\frac{1}{p^* \left(\beta_{m+1} -1 \right)}} \leq M_0 \left( 1+ \int_{\Om} u^{\beta_1 p^*} dx \right)^{\frac{1}{p^* \left(\beta_1 -1 \right)}}.\]
    This implies that $u\in L^p (\Om)$ for every $1\leq p < \infty$. Finally, applying a limiting argument, we infer that $u\in L^{\infty}(\Om)$ with the uniform bound
    \[ \|u\|_{\infty} \leq M_0 \left( 1+ \int_{\Om} u^{\beta_1 p^*} dx \right)^{\frac{1}{p^* (\beta_1 -1)}}.\]
    This concludes the proof in the case $r^* = p^* -1$.

    Now, we turn to the case $r^* = q_s^* -1$. In this setting, we define the function $\phi$ as follows:
    \[\phi(t) := \begin{cases}
        0, \text{ if } t < 0,\\[1mm]
        t^{(\ba -1)q+1}, \text{ if } 0 \leq t \leq T,\\[1mm]
        \left( (\ba-1)q+1 \right) T^{(\ba -1)q} (t-T) +T^{(\ba-1)q+1}, \text{ if } t > T
    \end{cases}\]
    and retain the definitions of $\phi_1$ and $\phi_2$ as before. Proceeding analogously to the previous case and interchanging the roles of $p$ with $q$ and $p^*$ with $q_s^*$, we obtain the estimates
    \[J_1 \geq 0 \text{ and } J_2 \geq [\phi_2 (u)]_{s,q}^q.\]
    Moreover, with the same definition of $\psi$, here we get $\phi_2 (t) = \frac{1}{\beta} \left((\beta -1)q +1 \right)^{1/q} \psi (t)$, which ultimately yields $u\in L^{\beta_1 q_s^*}(\Om)$, where $\beta_1 = \frac{q_s^* -1}{q}+1$. Finally, similar to the previous case, the desired $L^{\infty}$-estimate of $u$, given by
    \[ \|u\|_{\infty} \leq M_0 \left( 1+ \int_{\Om} u^{\beta_1 q_s^*}\ dx \right)^{\frac{1}{q_s^* (\beta_1 -1)}},\]
    where $M_0 = M_0(p, q,s,\Om,C_0)$, follows once again by employing the Moser iteration method and thereby completing the proof.
\end{proof}


\bibliographystyle{amsplain}
\bibliography{ref_arxiv}	

\end{document}